\documentclass[10pt]{amsart}

\usepackage[utf8]{inputenc}
\usepackage{color}
\usepackage{amsmath} 
\usepackage{amssymb} 
\usepackage{amsthm}
\usepackage{graphicx}
\usepackage{esint}
\usepackage[colorlinks=true,linkcolor=blue]{hyperref}

\theoremstyle{plain}

\newtheorem{prop}{Proposition}[section]

\newtheorem{lem}{Lemma}[section] \newtheorem{cor}{Corollary}[section]

\newtheorem{defi}{Definition}[section]
\theoremstyle{remark}
\newtheorem{rmk}{Remark}
\newcommand {\R} {\mathbb{R}} 
 \newcommand {\N} {\mathbb{N}}
 
\newcommand {\p} {\partial}
\newcommand {\dt} {\partial_t}

\newcommand {\dr} {\partial_r}

\newcommand {\ve} {v^{\epsilon}}
\newcommand {\D} {\Delta}

\newcommand {\supp} {\text{supp}}

\usepackage{mathtools}

\DeclarePairedDelimiter{\floor}{\lfloor}{\rfloor}

\DeclareMathOperator {\dist} {dist}
\DeclareMathOperator {\Ree} {Re}

\title[The Variable Coefficient Thin Obstacle Problem]{The Variable Coefficient Thin Obstacle Problem: Carleman Inequalities}
\author{Herbert Koch}
\author{Angkana R\"uland }
\author{Wenhui Shi}

\address{
Mathematisches Institut, Universit\"at Bonn, Endenicher Allee 60, 53115 Bonn, Germany }
\email{koch@math.uni-bonn.de}

\address{
Mathematical Institute of the University of Oxford, Andrew Wiles Building, Radcliffe Observatory Quarter, Woodstock Road, OX2 6GG Oxford, United Kingdom }
\email{ruland@maths.ox.ac.uk}

\address{
Mathematisches Institut, Universit\"at Bonn, Endenicher Allee 64, 53115 Bonn, Germany  }
\email{wenhui.shi@hcm.uni-bonn.de}

\begin{document}

\begin{abstract}
In this article we present a new strategy of addressing the (variable coefficient) thin obstacle problem. Our approach is based on a (variable coefficient) Carleman estimate. This yields semi-continuity of the vanishing order, lower and uniform upper growth bounds of solutions and sufficient compactness properties in order to carry out a blow-up procedure. Moreover, the Carleman estimate implies the existence of homogeneous blow-up limits along certain sequences and ultimately leads to an almost optimal regularity statement. As it is a very robust tool, it allows us to consider the problem in the setting of Sobolev metrics, i.e. the coefficients are only $W^{1,p}$ regular for some $p>n+1$.\\
These results provide the basis for our further analysis of the free boundary, the optimal ($C^{1,1/2}$-) regularity of solutions and a first order asymptotic expansion of solutions at the regular free boundary which is carried out in a follow-up article, \cite{KRSI}, in the framework of $W^{1,p}$, $p>2(n+1)$, regular coefficients. 
\end{abstract}



\thanks{A.R. acknowledges that the research leading to these results has received funding from the European Research Council under the European Union's Seventh Framework Programme (FP7/2007-2013) / ERC grant agreement no 291053 and a Junior Research Fellowship at Christ Church.
W.S. is supported by the Hausdorff Center of Mathematics.}
\maketitle

\tableofcontents

\section{Introduction}

In this article we present a new, very robust strategy of analyzing solutions of the (variable coefficient) \emph{thin obstacle} or \emph{Signorini problem}. Let $a^{ij}: B_1^+ \rightarrow \R^{(n+1)\times (n+1)}_{sym}$ be a symmetric, uniformly elliptic tensor field which is $W^{1,p}$, $p>n+1$, regular and let $B_1^+:=\{x\in B_1\subset \R^{n+1}| \ x_{n+1}\geq 0\}$ denote the upper half-ball. Then consider local minimizers of the constrained Dirichlet energy:
\begin{align*}
J(w) = \int\limits_{B_1^+} a^{ij} (\p_i w) (\p_j w) dx,
\end{align*}
where we use the Einstein summation convention and assume that
\begin{align*}
 w\in \mathcal{K}:=\{v\in H^1(B_1^+)| \ v \geq 0 \mbox{ on } B_1':= B_1^+\cap \{x_{n+1}=0\} \}.
\end{align*}
Thus, local minimizers of this variational problem solve a uniformly elliptic divergence form equation in the interior of the upper half-ball, on this set they are ``free''. However, on the codimension one surface $B_1'$ they obey the convex constraint $w\geq 0$. In this sense the obstacle is ``thin''.\\

Due to results of Caffarelli \cite{Ca79}, Kinderlehrer \cite{Ki81} (who work in the setting with $C^{1,\gamma}$ coefficients) and Uraltseva \cite{U87}, local minimizers are $C^{1,\alpha}$ regular for some $\alpha \in (0,1/2]$ and satisfy a second order elliptic equation with \emph{Signorini} (or \emph{complementary}) boundary conditions: 
\begin{equation}
\begin{split}
\label{eq:varcoef'}
\p_i  a^{ij} \p_j  w & = 0 \mbox{ in } B^+_{1}, \\
 w \geq 0,   -a^{n+1,j} \p_{j}w \geq 0 \mbox{ and } w ( a^{n+1,j} \p_j w) &= 0 \mbox{ on } B_{1}'.
\end{split}
\end{equation}
In the sequel we study (\ref{eq:varcoef'}) with the aim of obtaining optimal regularity estimates for its solutions as well as a better understanding of the free boundary, i.e. the set $\Gamma_w = \partial_{B'_1} \{x\in B_1'| \ w(x)>0 \}$, which separates the coincidence set, $\Lambda_w:=\{x\in B_1'| \ w(x)=0\}$, in which the solution coincides with the obstacle $w=0$, from the positivity set, $\Omega_w:=\{x\in B_1'| \ w(x)>0\}$. 

\subsection{Main results}
In studying the variable coefficient thin obstacle problem (\ref{eq:varcoef'}) in a low regularity framework (we only assume that the coefficients $a^{ij}$ are in an appropriate Sobolev class), we introduce \emph{Carleman estimates} as a key ingredient of obtaining first information on the solutions. Here the Carleman estimate replaces a variable coefficient frequency function approach. \\

Carleman estimates have the advantage of being very flexible with respect to perturbations:
After deriving a constant coefficient Carleman inequality, it is often possible to deduce a variable coefficient analogue by perturbative techniques. Comparing the Carleman estimate with monotonicity arguments, we note that monotonicity of the frequency function (for harmonic functions) is equivalent to logarithmic convexity of the $L^2$ norm. It depends on rigid calculations. In contrast, Carleman inequalities provide relaxed convexity statements (c.f. Corollary \ref{cor:consequence_Carl}). The conjugated equation is central (c.f. (\ref{eq:op})), and one stays in the more flexible context of PDEs.\\

Making use of the Carleman estimate, we obtain the upper semi-continuity of the vanishing order (c.f. Proposition \ref{prop:semi_cont}), lower and uniform upper growth bounds for solutions (c.f. Corollary \ref{cor:growth1} and Lemma \ref{lem:growthimp}),
compactness of $L^2$ normalized blow-up sequences (c.f. Proposition \ref{prop:doubling}). 
Moreover, we obtain that along certain sequences, the blow-up limits are homogeneous (c.f. Proposition \ref{lem:homo}). This in particular allows us to classify the blow-up limits when the vanishing order is less than $2$. Combining this information then allows us to prove the following almost optimal regularity result:

\begin{prop}[Almost optimal regularity]
\label{prop:almost_opti1}
Let $a^{ij}: B_1^+ \rightarrow \R^{(n+1)\times (n+1)}_{sym}$ be a uniformly elliptic, symmetric $W^{1,p}$ tensor field with $p \in (n+1,\infty]$.
Assume that $w$ is a solution of the variable coefficient thin obstacle problem (\ref{eq:varcoef'}). Then, with $\gamma = 1 - \frac{n+1}{p}$,
\begin{align*}
&|\nabla w(x) - \nabla w(y)| \\
&\leq \left\{ 
\begin{array}{ll}
C(\gamma) \left\| w \right\|_{L^2(B_1^+)} |x-y|^{\gamma} &\mbox{ for } p \in (n+1,2(n+1)),\\
C\left\| w \right\|_{L^2(B_1^+)} |x-y|^{1/2}\ln(|x-y|)^{2} & \mbox{ for } p \in [ 2(n+1),\infty],
\end{array}
\right.
\end{align*}
for all $x,y \in B_{\frac{1}{2}}^+$. Apart from the dependence on $\gamma$ for the first case, the constants $C$ are also functions of $\left\|a^{ij}\right\|_{W^{1,p}(B_1^+)}, p, n$.
\end{prop}

In particular, this improves the regularity estimates of Caffarelli \cite{Ca79}, Kinderlehrer \cite{Ki81} and Uraltseva \cite{U87}, by coming logarithmically close to the expected optimal threshold of $C^{1,1/2}$ regularity if $p\geq 2(n+1)$ (or even attaining the optimal $C^{1,\gamma}$ regularity if $p\in(n+1,2(n+1))$).\\

The Carleman estimate however permits us to obtain more information: The existence of homogeneous blow-up solutions directly allows us to classify the lowest possible vanishing rate at free boundary points without invoking the Friedland-Hayman inequality \cite{FH76} (c.f. the proof of Proposition \ref{prop:homo2}). Since blow-up solutions satisfy the constant coefficient thin obstacle problem,
it is not surprising that this lowest blow-up homogeneity coincides with the lowest possible homogeneity of solutions of the constant coefficient thin obstacle problem, $\kappa = 3/2$. As in the case of the constant coefficient problem, we show that there is a gap to the next possible blow-up homogeneity which is, $\kappa=2$.\\

In our follow-up article \cite{KRSI}, we use this and the upper semi-continuity of the mapping $\Gamma_w \ni x \mapsto \kappa_x$ to separate the free boundary into
$$\Gamma_w = \Gamma_{3/2}(w)\cup \bigcup\limits_{\kappa \geq 2}\Gamma_{\kappa}(w),$$  
where $\Gamma_{3/2}(w)$ is a relatively open set of the free boundary, the so-called \emph{regular} free boundary and $\bigcup\limits_{\kappa \geq 2}\Gamma_{\kappa}(w)$ consists of all free boundary points which have a higher order of vanishing. Working in the framework of $W^{1,p}$ metrics with $p\in(2(n+1),\infty]$, we prove that the regular free boundary is locally a $C^{1,\alpha}$ graph for some $\alpha\in(0,1)$ in \cite{KRSI}. This then allows us to improve the almost optimal regularity result from Proposition \ref{prop:almost_opti1} to an optimal regularity result. Moreover, we identify the leading term in the asymptotic expansion of solutions of (\ref{eq:varcoef'}) at the regular free boundary and provide explicit error bounds.

\subsection{Context and literature}
Let us comment on the context of our problem: Apart from the (constant coefficient) two-dimensional problem, which was completely solved by \cite{Le72} and \cite{Ri78}, and despite various impressive results on the higher dimensional problem \cite{Ca79}, \cite{Fr77}, \cite{Ki81}, \cite{U87}, the optimal regularity of the thin obstacle problem was only resolved relatively recently by Caffarelli et al. \cite{AC06}, \cite{ACS08}. In these seminal papers a frequency function was introduced as the key tool in studying solutions of the thin obstacle problem. \\

Following this there has been great progress in various directions for the \emph{constant} coefficient problem: Relying on frequency function methods, the ground breaking papers of Caffarelli et al. \cite{AC06}, \cite{ACS08} establish the optimal regularity of solutions, as well as the $C^{1,\alpha}$ regularity of the regular free boundary. This has been further extended to the related obstacle problem for the fractional Laplacian \cite{Si}, \cite{CSS} and parabolic analogues \cite{DGPT}. Relying neither on frequency functions nor on comparison arguments, Andersson \cite{An} has shown that similar results also hold for the full Lam\'e system. Recently, Koch, Petrosyan and Shi \cite{KPS} as well as De Silva and Savin \cite{DSS14} have proved smoothness (\cite{KPS} proved analyticity) of the regular free boundary. Moreover, Garofalo and Petrosyan \cite{GP09} give a structure theorem for the singular set of the thin obstacle problem. \\

The \emph{variable} coefficient thin obstacle problems is much less understood. Here the best regularity result in the literature in the setting of Sobolev regularity is given by Uraltseva's $C^{1,\alpha}$ regularity result: In \cite{U87} she proves that for the thin obstacle problem with a $W^{1,p}$, $p>n+1$ metric, $a^{ij}$, there exists a Hölder coefficient $\alpha \in (0,1/2]$, depending only on the Sobolev exponent $p$ and the ellipticity constants of $a^{ij}$, such that the corresponding solution is in $C^{1,\alpha}$ (c.f. also \cite{U89}). Recently there has been important progress in deriving an improved understanding in the low regularity setting: In \cite{G09} Guillen derives \emph{optimal} regularity results for solutions of the variable coefficient thin obstacle problem. He works in the setting of $C^{1,\gamma}$, $\gamma>0$ coefficients. This was generalized by Garofalo and Smit Vega Garcia \cite{GSVG14} to Lipschitz continuous coefficients by using the frequency function approach. In a recent work in progress Garofalo, Petrosyan and Smit Vega Garcia  \cite{GPSVG15} further extend this result and prove Hölder continuity of the regular free boundary.

\subsection{Difficulties and strategy}
Let us elaborate a bit further on the main difficulties in investigating the variable coefficient thin obstacle problem: The central problem that has to be overcome and that is reflected in all stages of our arguments is the low regularity of the metric. This requires robust tools. \\

In this first part of our discussion of the variable coefficient thin obstacle problem, Carleman estimates, which are well-known from unique continuation and the study of inverse problems \cite{Carl}, \cite{JK}, \cite{KT1}, \cite{I3}, \cite{Rue14}, are used to handle this low regularity setting: Although they are usually employed in the setting of Lipschitz metrics (as for instance the unique continuation principle in general fails for less regular coefficients), it is possible to extend these to the low regularity framework we are interested in.  \\

This allows us to carry out a blow-up argument and exploit information from the constant coefficient setting. For this we argue in two steps:
\begin{itemize}
\item \emph{A Carleman estimate for solutions of the variable coefficient thin obstacle problem (\ref{eq:varcoef}).} This yields sufficiently strong compactness properties in order to carry out a blow-up procedure at the free boundary points. In particular, doubling inequalities (c.f. Proposition \ref{prop:doubling}) are immediate consequences of the Carleman estimate.
\item \emph{A blow-up procedure.} The good compactness properties deduced in the previous step permit to carry out a blow-up procedure with a non-trivial limit satisfying a \emph{constant} coefficient thin obstacle problem (c.f. Proposition \ref{prop:blowup}). Thus, (for an appropriately normalized metric $a^{ij}$) the blown-up solution is of the form:
\begin{equation}
\begin{split}
\label{eq:constcoef}
\D w & = 0\mbox{ in } B^{+}_{1}, \\
w\geq 0,  -\p_{{n+1}}w\geq 0 \mbox{ and } w\p_{{n+1}}w&= 0 \mbox{ on } B_{1}'.
\end{split}
\end{equation}
However, the Carleman estimate yields further information: It is possible to show that the there are blow-up sequences such that the blow-up limits are homogeneous solutions of (\ref{eq:constcoef}). Moreover, the homogeneity is given by the order of vanishing (c.f. Proposition \ref{lem:homo}). For a solution of the variable coefficient thin obstacle problem this then allows to exploit the existing information on solutions of the constant coefficient thin obstacle problem and thus, for instance, obtain an almost optimal regularity result (c.f. Proposition \ref{prop:almost}).
\end{itemize}

\subsection{Organization of the paper}
Let us finally comment on the structure of the remainder of the article: In the next section, we first briefly recall auxiliary results (c.f. Proposition \ref{prop:change}) and then introduce our notational conventions in Section \ref{sec:not}. Following this, Section \ref{sec:Carl} is dedicated to our Carleman estimate, Proposition \ref{prop:varmet}. In particular, we derive an important corollary, Corollary \ref{cor:consequence_Carl}, from it. In Section \ref{sec:reg} we then deduce crucial consequences of the Carleman estimate: We derive compactness properties (c.f. Proposition \ref{prop:doubling}), carry out a blow-up procedure (c.f. Proposition \ref{prop:blowup}) and prove the existence of homogeneous of blow-up solutions (c.f. Proposition \ref{lem:homo}). This is then applied in proving the first (almost optimal) regularity result, Proposition \ref{prop:almost}. \\

In our second article dealing with the thin obstacle problem we then use these results to further analyze the regular free boundary, to derive optimal regularity estimates and to obtain a first order expansion of solutions at the regular free boundary.

\section{Preliminaries}

\subsection{Auxiliary results and assumptions}
\label{sec:aux}

In the sequel we recall certain auxiliary results which allow us to reformulate and simplify the problem (\ref{eq:varcoef'}) in a particularly useful way. Moreover, we collect all our assumptions on the involved quantities.\\

We start by recalling a (slight modification of a) result due to Uraltseva \cite{U89}, p.1183 which allows to simplify the complementary boundary conditions:
\begin{prop}
\label{prop:change}
Let $a^{ij}:B_{1}^+ \rightarrow \R^{(n+1)\times (n+1)}_{sym}$ be a uniformly elliptic $W^{1,p}$, $p > n+1$, tensor field and $w\in H^1(B_1^+)$ be a solution to \eqref{eq:varcoef'}. Then for each point $x\in B_{\frac{1}{2}}'$ there exist a neighborhood, U(x), and a $W^{2,p}$ diffeomorphism
\begin{align*}
T:U(x)\cap B_{1}^+ \rightarrow V\cap B_{1}^+, \ x \mapsto T(x)=:y,
\end{align*}
such that in the new coordinates $\tilde{w}(y):= w(T^{-1}y)$ (weakly) solves a new thin obstacle problem with
\begin{equation}
\label{eq:varcoef2}
\begin{split}
\p_k b^{k\ell}(y)\p_{\ell} \tilde{w}=0 \mbox{ in } V\cap B_{1}^+,\\
\p_{n+1} \tilde{w} \leq 0, \ \tilde{w} \geq 0, \ \tilde{w}(\p_{n+1} \tilde{w} ) = 0 \mbox{ on } V\cap B_{1}',
\end{split}
\end{equation}
where $B(y)=(b^{k\ell}(y))=|\det (DT(x))|^{-1}(DT(x))^tA(x)DT(x)\big |_{x=T^{-1}(y)}$ satisfies 
\begin{align*}
b^{n+1,\ell}=0 \mbox{ on } B'_1 \mbox{ for all } \ell\in\{1,\ldots,n\},
\end{align*}
and remains a uniformly elliptic $W^{1,p}$ tensor field.
\end{prop}

Here the boundary data in (\ref{eq:varcoef'}) and (\ref{eq:varcoef2}) are interpreted in a distributional sense: It is known that one can see from the equation by local arguments that, if $ w \in H^1$ is a weak solution in the interior, then $a^{n+1,j} \partial_j w|_{x_{n+1}=0} \in H^{-1/2}$. 
This remains true in the new coordinates: $b^{n+1,n+1} \partial_{n+1} \tilde w|_{x_{n+1}=0}  \in H^{-1/2}$. Since $b^{n+1,n+1}$ is bounded from below and in $W^{1,p}$, the multiplication by its inverse defines a bounded operator on $H^{1/2}$ and $H^{-1/2}$. Thus, $\tilde{w} \partial_{n+1} \tilde{w} |_{x_{n+1}=0}$ defines a distribution on the boundary which we require to be zero. Vice versa: If $\tilde{w} \in H^1$, $\partial_k b^{kl} \partial_l \tilde{w}= 0 $ in $V \cap B_1^+$ and $\partial_{n+1} \tilde w \le 0$, $\tilde w \ge 0$ and $\tilde w \partial_{n+1} \tilde w=0$, then $\tilde w$ is a local solution to the thin obstacle problem. In any of the two formulations the regularity theory
of Uraltseva \cite{U87} applies.
\\

By a further (affine) change of coordinates we may, without loss of generality, assume that
\begin{align*}
|b^{k \ell}(x)-\delta^{k \ell}|\leq C|x|^{\gamma},
\end{align*}
for $\gamma = 1 - \frac{n+1}{p}$ and $\delta^{k \ell}:= \left\{ \begin{array}{ll} 1 & \mbox{ for } k=\ell, \\ 0 & \mbox{ else}, \end{array} \right.$ denotes the Kronecker delta.
In the sequel, for convenience, we will always assume that we are in a sufficiently small coordinate patch such that our thin obstacle problem is formulated in this way. This allows us to exploit the boundary conditions very efficiently. So, without loss of generality, we pass from (\ref{eq:varcoef'}) to considering

\begin{equation}
\begin{split}
\label{eq:varcoef}
\p_i  a^{ij} \p_j  w & = 0 \mbox{ in } B^+_{1}, \\
w \geq 0  \ -\p_{{n+1}}w  \geq 0 \mbox{ and } w (\p_{n+1} w)&= 0 \mbox{ on } B_{1}',
\end{split}
\end{equation}
with 
\begin{itemize}
\item[(A0)] $\left\| w \right\|_{L^2(B_1^+(0))}=1$,
\item[(A1)] $a^{i, n+1}(x',0)=0 \mbox{ on } \R^{n} \times \{0\}$ for $i=1,\ldots, n$,
\item[(A2)] $a^{ij}$ is symmetric and uniformly elliptic with eigenvalues in the interval $[1/2, 2]$,
\item[(A3)] $a^{ij}\in W^{1,p}(B_1^+(0))$ for some $p\in (n+1,\infty]$,
\item[(A4)] $a^{ij}(0)= \delta^{ij}$.
\end{itemize}
Moreover, throughout the paper we assume that the solution, $w$, of (\ref{eq:varcoef}) satisfies
\begin{itemize}
\item[(W)] there exists $\alpha \in (0,\frac{1}{2}]$ such that $w\in C^{1,\alpha}$.
\end{itemize}
Assumption (A0) is a normalization, which is no restriction of generality, since positive multiples of solutions are solutions. Property (A1) may be assumed because of Proposition \ref{prop:change}. By Morrey's inequality and assumption (A3)
\[   |a^{ij}(x) -\delta^{ij}| \leq C\left\| \nabla a^{ij}\right\|_{L^p(B_1^+(0))}|x|^{\gamma},\]  where $\gamma = 1- \frac{n+1}{p}$ and hence the condition (A2) on the eigenvalues holds after rescaling if necessary.
The regularity assumption (W) does not pose any restrictions on the class of solutions, as Uraltseva \cite{U87} proves that this property is true for any $H^1$ solution of (\ref{eq:varcoef}). We however stress that most of the paper is independent of this regularity assumption in the sense that assumption (W) is only invoked in Proposition \ref{prop:homo2}. Apart from this, all the other necessary regularity results are proved ``by hand'' in the paper.

\subsection{Notation}
\label{sec:not}
In this subsection we briefly introduce some notation. We set:
\begin{itemize}
\item $\R^{n+1}_+ := \{x\in \R^{n+1}| \ x_{n+1}\geq 0\}$,  $\R^{n+1}_- := \{x\in \R^{n+1}| \ x_{n+1}\leq 0\}$.
\item Let $x_0=(x_0',0) \in \R^{n+1}_+$. For an upper half-ball of radius $r>0$ around $x_0$ we write $B_r^+(x_0):=\{x\in \R^{n+1}_+| \ |x-x_0|< r \}$; the projection onto the boundary of $\R^{n+1}_+$ is correspondingly denoted by $B_r'(x_0):=\{x\in \R^{n}| \ |x-x_0|< r \}$. If $x_0 = (0,0)$ we also write $B_r^+$ and $B_r'$. Analogous conventions are used for balls in the lower half sphere: $B^-_r(x_0)$. Moreover, we use the notation $B_r''(x_0) = B_r'(x_0)\cap\{x_n=0\}$.
\item Annuli around a point $x_0=(x_0',0)$ in the upper half-space with radii $0<r<R<\infty$ as well as their projections onto the boundary of $\R^{n+1}_+$ are denoted by by $A_{r,R}^+(x_0):= B_{R}^+(x_0)\setminus B_{r}^+(x_0)$ and $A_{r,R}'(x_0):= B_{R}'(x_0)\setminus B_{r}'(x_0)$ respectively. For annuli around $x_0=(0,0)$ we also omit the center point. Furthermore, we set $A_{r,R}(x_0):= A_{r,R}^+(x_0)\cup A_{r,R}^-(x_0)$.
\item $\Omega_w:= \{x\in \R^{n}\times \{0\}| \ w(x)>0\}$ denotes the positivity set. 
\item $\Gamma_w:=\partial_{B'_1}  \Omega_w$ is the free boundary.
\item $\Lambda_w:= B_1'\setminus \Omega_w$ is the coincidence set which we also denote by $\Lambda_w$.
\item  We use the symbol $A \lesssim B$ to denote that there exists an only dimension dependent constant, $C=C(n)$, such that $A\leq C(n)B$. Similar conventions are used for the symbol $\gtrsim$.
\end{itemize}

\section{A Carleman Estimate}
\label{sec:Carl}

\subsection{A Carleman estimate and variations of it}
In this section we introduce a central tool of our argument: The Carleman estimate from Proposition \ref{prop:varmet} allows us to obtain compactness properties for blow-up solutions (c.f. Proposition \ref{prop:blowup}), to deduce the existence of \emph{homogeneous} blow-up solutions (c.f. Proposition \ref{lem:homo}) and to derive the openness of the regular part of the free boundary (c.f. Proposition \ref{prop:semi_cont}). It is a very flexible and robust tool replacing a variable coefficient frequency function argument.

\begin{prop}[Variable coefficient Carleman estimate]
\label{prop:varmet}
Let $n\geq 2$ and $0<\rho<r<1$. 
Let $a^{ij}: A_{\rho,r}^+ \rightarrow \R^{n\times n}_{sym}$ be a tensor field which satisfies 
\begin{itemize}
\item[(i)] $a^{ij}\in W^{1,n+1}(A_{\rho,r}^+)$, 
\item[(ii)] the off-diagonal assumption (A1),
\item[(iii)] the uniform ellipticity assumption (A2),
\item[(iv)] the following smallness condition: There exists $\delta = \delta(n)>0$ such that 
\begin{equation}
\label{eq:small}
\sup\limits_{\rho \leq \tilde{r} \leq r} \left\| \nabla a^{ij} \right\|_{L^{n+1}(A_{\tilde{r},2 \tilde{r}}^+)} \leq \delta.
\end{equation}
\end{itemize}
Assume that $w\in H^1(B_1^+)$ with $\supp(w) \subset \overline{A^+_{\rho, r}}$ is a weak solution of the divergence form equation
\begin{equation}
\label{eq:problem}
\begin{split}
\partial_i a^{ij}\partial_j w &=f \text{ in } A_{\rho, r}^+,\\
w \geq 0, \ \p_{n+1}w \leq 0, \ w \p_{n+1}w &=0 \text{ on } A'_{\rho, r},
\end{split}
\end{equation}
where $f:A_{\rho,r}^+ \rightarrow \R$ is an in $\overline{A_{\rho,r}^+}$ compactly supported $L^2(A_{\rho,r}^+)$ function.
Let $\phi$ be the following radial weight function: 
$$\phi(x)=\tilde{\phi}(\ln|x|) \text{ with } \tilde{\phi}:\R\rightarrow \R, \ \tilde{\phi}(t)=-t+c_0\left(t\arctan t-\frac{1}{2}\ln(1+t^2)\right),$$
where $0<c_0\ll 1$ is an arbitrarily small but fixed constant. 
Then for any $\gamma \in (0,1)$ and any $\tau>1$ we have
\begin{equation}
\label{eq:vCarl}
\begin{split}
&\tau^{\frac{3}{2}} \left\|  e^{\tau \phi}|x|^{-1} (1+\ln(|x|)^2)^{-\frac{1}{2}}  w \right\|_{L^2(A_{\rho, r}^+)} 
+  \tau^{\frac{1}{2}}\left\| e^{\tau \phi} (1+\ln(|x|)^2)^{-\frac{1}{2}}  \nabla w \right\|_{L^2(A_{\rho,r}^+)}\\
&\leq c_0^{-1}C(n) \left(  
 \tau^2 C(a^{ij}) \left\| e^{\tau\phi}|x|^{\gamma-1} w \right\|_{L^2(A_{\rho,r}^+)}  + \left\|e^{\tau \phi} |x| f\right\|_{L^2(A_{\rho,r}^+)} \right),
\end{split}
\end{equation}
where 
\begin{align}
\label{eq:coef_dep}
C(a^{ij}) = \sup\limits_{A_{\rho,r}^+} \left| \frac{a^{ij}(x)-\delta^{ij}}{|x|^{\gamma}} \right| + \sup\limits_{\rho\leq \tilde{r}\leq r/2} \left\| |x|^{-\gamma} \nabla a^{ij} \right\|_{L^{n+1}(A_{\tilde{r},2\tilde{r}}^+)}.
\end{align}
\end{prop}

\begin{rmk}[Restrictions on $\tau$]
\label{rmk:tau}
In order to derive the Carleman estimate, it would not have been necessary to assume that $\tau>1$. In fact the estimate is valid for arbitrary $\tau>0$ if we replace the right hand side of (\ref{eq:vCarl}) by
\begin{align*}
&\tau^{\frac{3}{2}} \left\|  e^{\tau \phi}|x|^{-1} (1+\ln(|x|)^2)^{-\frac{1}{2}}  w \right\|_{L^2(A_{\rho, r}^+)} 
+  \tau^{\frac{1}{2}}\left\| e^{\tau \phi} (1+\ln(|x|)^2)^{-\frac{1}{2}}  \nabla w \right\|_{L^2(A_{\rho,r}^+)}\\
&\leq c_0^{-1}C(n) \left(  
 \max\{1,\tau^2 \} C(a^{ij}) \left\| e^{\tau\phi}|x|^{\gamma-1} w \right\|_{L^2(A_{\rho,r}^+)}  + \left\|e^{\tau \phi} |x| f\right\|_{L^2(A_{\rho,r}^+)} \right.\\
& \quad \left. + \tau^{\frac{1}{2}}\left\| e^{\tau \phi}|x|^{-1} (1 + \ln(|x|)^4)^{-\frac{1}{2}}w \right\|_{L^2(A_{\rho,r}^+)} \right).
\end{align*}
Choosing $0<r<1$ depending on $\tau$ sufficiently small, then allows to absorb the error terms in this case as well (c.f. Remark \ref{rmk:Carl_application}).
\end{rmk}

\begin{rmk}[Applications]
\label{rmk:Carl_application}
In the sequel we will apply the Carleman estimate (\ref{eq:vCarl}) with a metric $a^{ij}$ such that $a^{ij}(0)=\delta^{ij}$ and $a^{ij}\in W^{1,p}(B_1^+)$, $p\in (n+1,\infty]$. Thus, by Morrey's and Hölder's inequalities, for $\gamma=1-\frac{n+1}{p}$, $C(a^{ij})$ is bounded by a constant which depends only on $n$ and $\|\nabla a^{ij}\|_{L^p(B_1^+)}$.

In our applications, we will always consider $r\leq R_0=R_0(\tau,n,p,\|\nabla a^{ij}\|_{L^p(B_1^+)})$ sufficiently small, such that the smallness condition  \eqref{eq:small} is satisfied and moreover, the first term on the right hand side of \eqref{eq:vCarl} can be absorbed by the left hand side. More precisely, we consider $R_0$ such that
\begin{align*}
C(a^{ij})R_0^{\gamma}\leq \delta, \quad \tau ^{\frac{1}{2}}c_0^{-1}C(n)C(a^{ij})R_0^\gamma |\ln R_0| \leq 1/4.
\end{align*}
Then the Carleman inequality \eqref{eq:vCarl} can be rewritten as
\begin{equation}
\label{eq:vCarl_modified}
\begin{split}
&\tau^{\frac{3}{2}} \left\|  e^{\tau \phi}|x|^{-1} (1+\ln(|x|)^2)^{-\frac{1}{2}}  w \right\|_{L^2(A_{\rho, r}^+)} 
+  \tau^{\frac{1}{2}}\left\| e^{\tau \phi} (1+\ln(|x|)^2)^{-\frac{1}{2}}  \nabla w \right\|_{L^2(A_{\rho,r}^+)}\\
&\leq c_0^{-1}C(n) 
 \left\|e^{\tau \phi} |x| f\right\|_{L^2(A_{\rho,r}^+)} \quad \text{for } 0<\rho<r\leq R_0 \mbox{ and } \tau \geq 1.
\end{split}
\end{equation}
\end{rmk}

\begin{rmk}[Carleman estimate for $n=1$]
\label{rmk:Carl_n=1}
In one dimension it is possible to prove a similar Carleman estimate as above. However, slight modifications are necessary. Instead of working with the $\dot{W}^{1,2}$ semi-norm (which would be the scale invariant norm in $(1+1)$-dimensions) of the metric and the smallness condition (\ref{eq:small}), we consider $L^p$ norms, $\left\| \nabla a^{ij} \right\|_{L^p}$ with $p>2$. Moreover, all Sobolev embedding arguments are then replaced by interpolation inequalities of the form
\begin{equation}
\label{eq:interpol}
\left\| w \right\|_{L^{\frac{2p}{p-2}}} \leq c(p,n) \left\| \nabla w \right\|_{L^2}^{\frac{2}{p}} \left\| w \right\|_{L^2}^{1-\frac{2}{p}}.
\end{equation}
This then yields the following Carleman estimate 
\begin{equation}
\label{eq:vCarl_2D}
\begin{split}
&\tau^{\frac{3}{2}} \left\|  e^{\tau \phi}|x|^{-1} (1+\ln(|x|)^2)^{-\frac{1}{2}}  w \right\|_{L^2(A_{\rho, r}^+)} 
+ \tau^{\frac{1}{2}}\left\| e^{\tau \phi} (1+\ln(|x|)^2)^{-\frac{1}{2}}  \nabla w \right\|_{L^2(A_{\rho,r}^+)}\\
&\leq c_0^{-1}C(n,p) \left(  
 \tau^2 \left\| \nabla a^{ij} \right\|_{L^p} \left\| e^{\tau\phi}|x|^{\gamma-1} w \right\|_{L^2(A_{\rho,r}^+)} 
 + \left\|e^{\tau \phi} |x| f\right\|_{L^2(A_{\rho,r}^+)} \right).
\end{split}
\end{equation}
\end{rmk}

\begin{rmk}[Generalizations]
\begin{itemize}
\item With slight modifications, the Carleman estimate is valid for more general (radial) pseudoconvex weight functions, $\phi(x) = \tilde{\phi}(\ln(|x|))$. In this more general setting (but with $\tilde{\phi}(t)\sim -t$ and $| \tilde{\phi}'(t)| \sim 1$ asymptotically as $t \rightarrow -\infty$) the Carleman estimate reads
\begin{equation}
\label{eq:vCarl'}
\begin{split}
& \tau^{\frac{3}{2}} \left\|  e^{\tau \phi}|x|^{-1} \tilde{\phi}' (\tilde{\phi}'')^{\frac{1}{2}}  w \right\|_{L^2(A_{\rho, r}^+)} + \tau^{\frac{1}{2}} \left\|  e^{\tau \phi}   (\tilde{\phi}'')^{\frac{1}{2}}  \nabla w \right\|_{L^2(A_{\rho, r}^+)} \\
&\leq    c_0^{-1}C(n)\left( \tau^2  C(a^{ij}) \left\| e^{\tau\phi}|x|^{\gamma -1} w \right\|_{L^2(A_{\rho,r}^+)} 
+  \left\|e^{\tau \phi}|x| f\right\|_{L^2(A_{\rho,r}^+)} \right),
\end{split}
\end{equation}
where $C(a^{ij})$ is the constant from (\ref{eq:coef_dep}).
\item As we aim at absorbing the error terms on the right hand side of (\ref{eq:vCarl}) or (\ref{eq:vCarl'}) into the left hand side of the inequality if $\tau>1$ is uniformly bounded by some $\tau_0<\infty$, it is in general not possible to work with ``linear'' weight functions in (\ref{eq:vCarl'}). As a consequence, in the sequel we will always consider strictly pseudoconvex weight functions with the exception of Lemma \ref{lem:varmet1}.
\item In contrast to the setting of unique continuation, where the metrics under consideration are \emph{Lipschitz}, our metric is \emph{only $W^{1,p}$, $p>n+1$, regular}. On the one hand, this requires additional care in the commutator estimates of the Carleman inequality. On the other hand, the low regularity of the metric also implies that we cannot hope for the unique continuation principle to hold in the setting of (\ref{eq:varcoef}). Hence, the Carleman estimate can, for instance, \emph{not} exclude free boundary points with infinite order of vanishing. On a technical level these observations are manifested in the error contributions on the right hand side of the inequality (\ref{eq:vCarl}) which carry \emph{higher} powers of $\tau$ than the analogous contributions on the left hand side.
\item It is possible to derive an analogous Carleman inequality for the setting of the \emph{interior} thin obstacle problem and the case in which inhomogeneities are present in the thin obstacle problem. For a detailed discussion of this we refer to Section 5 in \cite{KRSI}.
\end{itemize}
\end{rmk}

Before proving the Carleman estimate, we state $H^2$ estimates for the associated conjugated operator in dyadic annuli which will be proved in Section \ref{sec:H2} at the end of this Section. 

\begin{lem}
\label{lem:H2}
Let $n \geq 2$. Then there exists $\delta= \delta(n)>0$ such that the following is true:
Let $a^{ij}:B_{1}^+ \rightarrow \R^{(n+1)\times (n+1)}_{sym}$ be a uniformly elliptic, symmetric tensor field which satisfies (\ref{eq:small}).
Further, assume that $v: B_1^+ \rightarrow \R$ with $v\in H^{1}(B_1^+)$ is a weak solution of
\begin{equation}
\label{eq:H2_est}
\begin{split}
&\p_i a^{ij} \p_j v + \tau^2 (\p_{i} \phi) a^{ij} (\p_j \phi)v 
- \tau a^{ij}(\p_i \phi) \p_j v \\
& \quad \quad \quad - \tau (\p_j \phi)\p_i(a^{ij}v)
- \tau a^{ij} (\p_{ij}^2 \phi) v = f \mbox{ in } B_{1}^+, \\
&\p_{n+1} v \leq 0, \ v\geq 0, \ v\p_{n+1} v =0  \mbox{ on } B_{1}',\\
&v= 0 \mbox{ on } \partial B_1^+,
\end{split}
\end{equation}
where $\phi(x)= \tilde{\phi}(\ln(|x|))$ is defined as in Proposition \ref{prop:varmet} and $f:B_1^+ \rightarrow \R$ is an $L^2(B_1^+)$ function.
Then on each dyadic (half-)annulus $ A_m^{+}:=B_{2^{-m}}^+ \setminus B_{2^{-m-1}}^+$, $m\in \N$, $m \geq 1$, and for each $\tau>1$ 
\begin{equation}
\begin{split}
\label{eq:H2}
\left\| \nabla^2 v \right\|_{L^2(A_m^+)}  \lesssim &\ \tau^2  \left\| |x|^{-2} v \right\|_{L^2(A_{m-1}^+\cup A_m^+ \cup A_{m+1}^+)} + \left\| f \right\|_{L^2(A_{m-1}^+\cup A_m^+ \cup A_{m+1}^+)}.
\end{split}
\end{equation}
\end{lem}
 
\begin{rmk}[Modifications for $n=1$] Analogous to the Carleman estimate, Lemma \ref{lem:H2} has a similar extension to $n=1$. In that case, we replace the smallness condition (\ref{eq:small}) and all Sobolev embedding arguments by interpolation arguments as in (\ref{eq:interpol}).
\end{rmk}

\begin{proof}[Proof of Proposition \ref{prop:varmet}]
We interpret the equation as a perturbation of the Laplacian
\begin{align}
\label{eq:perturb}
\D w = \p_{i}(a^{ij} - \delta^{ij})\p_j w + f,
\end{align}
and argue in two steps: First, we derive a Carleman inequality for the Laplacian with Signorini boundary conditions. In a second step, we explain how to treat the right hand side of (\ref{eq:perturb}). However, before coming to this, we recall that the boundary contributions in (\ref{eq:problem}) are well-defined as distributions.\\

\emph{Step 1: Regularity of $w$.}
In order to make sense of the boundary contributions, we refer to the discussion following (\ref{eq:varcoef2}). Moreover, using the $H^2$ regularity of $w$ (c.f. Lemma \ref{lem:H2}), we may also deal with second derivatives of $w$ in $L^2(B_1^+)$ as well as $L^2$ gradient contributions on the boundary. An alternative method to deal with, for example, boundary contributions would have been a regularization mechanism (c.f. Remark \ref{rmk:regularization})\\

\emph{Step 2: Carleman for the Laplacian with Signorini boundary conditions.} We will show that
\begin{equation}
\label{eq:vCarl00}
\begin{split}
&\tau^{\frac{3}{2}} \left\|  e^{\tau \phi}|x|^{-1} (1+\ln(|x|)^2)^{-\frac{1}{2}}  w \right\|_{L^2(A_{\rho, r}^+)} +  \tau^{\frac{1}{2}}\left\| e^{\tau \phi} (1+\ln(|x|)^2)^{-\frac{1}{2}}  \nabla w \right\|_{L^2(A_{\rho,r}^+)}\\
&\leq  c_0^{-1}C(n)\left\|e^{\tau \phi}|x|\D w \right\|_{L^2(A_{\rho,r}^+)}.
\end{split}
\end{equation}
As in the article of Koch and Tataru, \cite{KT1}, we carry out a change into radial conformal coordinates.\\
To this end, we first consider the $(n+1)$-dimensional Laplacian in polar coordinates
\begin{align*}
\p_{r}^2+\frac{n}{r}\p_{r}+\frac{1}{r^2}\D_{S^{n}},
\end{align*}
and then carry out a change into radial conformal coordinates. Setting $r=e^{t}$, yields $\dr = e^{-t}\dt$, which in turn transforms the operator into 
\begin{align*}
e^{-2t}( \dt^2 + (n-1)\dt + \D_{S^{n}}).
\end{align*}
Conjugating with $e^{-\frac{n-1}{2}t}$ (which corresponds to setting $w=e^{- \frac{n-1}{2}t}\tilde{u}$) and multiplying the operator with $e^{2t}$, results in an operator on $\R\times S^{n}_+$ which is of the form
\begin{align}
\label{eq:op1}
\dt^2 - \frac{(n-1)^2}{4} + \D_{S^n}.
\end{align}
In the sequel, we only work with this modified operator and prove a Carleman estimate for it. Under the described change of coordinates and using the complementary boundary conditions combined with the $H^2$ regularity of $w$ (which follows from Lemma \ref{lem:H2}), the original Cartesian boundary conditions 
\begin{align*}
 w\geq 0, \ -\p_{n+1}w \geq 0 \mbox{ and } w \p_{n+1}w = 0 \mbox{ on } \R^n,
\end{align*}
now read
\begin{align*}
 \tilde{u} \geq 0, \  \nu\cdot \nabla_{S^n}\tilde{u} \geq 0 \mbox{ and } \tilde{u} (\nu\cdot \nabla_{S^n}\tilde{u}) = 0 \mbox{ on }  \R\times S^{n-1},
\end{align*}
where $\nu=(0,...,0,-1)$ denotes the outward unit normal to $S^{n}_+$. \\
In these new coordinates \eqref{eq:vCarl00} can be phrased as
\begin{align*}
&\tau^{\frac{3}{2}}\left\| e^{\tau \tilde{\phi}} (1+t^2)^{-\frac{1}{2}} \tilde{u} \right\|_{L^2((-\infty,0] \times S^n_+)}
+ \tau^{\frac{1}{2}}\left\| e^{\tau \tilde{\phi}} (1+t^2)^{-\frac{1}{2}} \nabla_{(t,S^n)} \tilde{u} \right\|_{L^2((-\infty,0] \times S^n_+)}\\
&\leq c_0^{-1}C(n) \left\| e^{\tau \tilde{\phi}}\left(\dt^2 - \frac{(n-1)^2}{4} + \D_{S^n}\right) \tilde{u} \right\|_{L^2((-\infty,0] \times S^n_+)}.
\end{align*}

We conjugate the operator (\ref{eq:op1}) with the only $t$-dependent weight, $e^{\tau \tilde{\phi}}$ (and in particular, we define $u=e^{\tau \tilde{\phi}}\tilde{u}$). This leads to the following (up to boundary contributions) ``symmetric'' and ``antisymmetric'' parts of the operator:
\begin{equation}
\label{eq:op}
\begin{split}
S & = \dt^2 + \tau^2(\dt \tilde{\phi})^2 - \frac{(n-1)^2}{4} + \D_{S^n},\\
A & = -2\tau (\dt \tilde{\phi})\dt - \tau \dt^2 \tilde{\phi}.
\end{split}
\end{equation}
In order to derive the desired Carleman estimate, we expand the $L^2$-norm of $(A+S)u$:
\begin{align*}
\left\| (S+A)u \right\|_{L^2(M)}^2 = \left\| S u \right\|_{L^2(M)}^2 + \left\| A u \right\|_{L^2(M)}^2 + 2(S u,A u)_{L^2(M)},
\end{align*}
where $M=\R_- \times S^n_+$. In order to bound this from below, we have to control the last contribution $2(S u,A u)_{L^2(M)}$. To this end, we integrate by parts and reformulate it as a commutator. As our ``symmetric'' and ``antisymmetric'' operators are only symmetric and antisymmetric up to boundary contributions, this also leads to boundary integrals:
\begin{align}
\label{eq:comm}
2(S u, A u)_{L^2(M)} = ([S,A] u,u)_{L^2(M)} + \mbox{ boundary contributions.}
\end{align}
A calculation shows that the boundary contributions are of the form
\begin{equation}
\label{eq:boundary_cont}
\begin{split}
- 8\tau \int\limits_{\R_- \times S^{n-1}} \dt \tilde{\phi} (\nu\cdot \nabla_{S^n}u )\dt u d\mathcal{H}^{n-1}dt - 4\tau \int\limits_{\R_- \times S^{n-1}} \dt^2 \tilde{\phi} (\nu\cdot \nabla_{S^n}u ) u d\mathcal{H}^{n-1}dt,
\end{split} 
\end{equation}
where $\nu=(0,\dots,0,-1)$ denotes the outer unit normal field associated with $S^{n}_+$. \\
We begin by discussing these boundary contributions:  Due to the radial dependence of $\tilde{\phi}$ (which in Cartesian variables implies that all the terms involving $\p_{n+1} \phi$ vanish on $\supp(w)\cap B_1'$), the integrands in (\ref{eq:boundary_cont}) correspond to contributions of the form
\begin{align*}
(x\cdot \nabla w )\p_{n+1}w \mbox{ and } w\p_{n+1}w,
\end{align*}
in Cartesian variables. By the complementary boundary conditions and the $H^2$ regularity of $w$, both of these vanish on $B_{1}'$: In fact, $w\in H^1(A_{\rho,r}')$ and thus $x\cdot \nabla w = 0$ almost everywhere where $w(x)=0$ (c.f. for instance Theorem 6.17 in \cite{LL01}). In particular, $x\cdot \nabla w=0$ for almost every $x \in \Lambda_w$. Combined with the vanishing of $\p_{n+1}w$ on $ \Omega_w$, this implies that the boundary contributions in the Carleman estimate vanish. \\
We hence proceed to the bulk commutator contribution in (\ref{eq:comm}): As the weight $\tilde{\phi}$ is sufficiently pseudoconvex, the commutator is bounded from below. Indeed,
\begin{align*}
([S,A]u,u)_{L^2(M)}&=  4\tau ((\dt^2 \tilde{\phi}) \dt u, \dt u)_{L^2(M)} + 4 \tau^3 ((\dt^2 \tilde{\phi})(\dt \tilde{\phi})^2 u, u)_{L^2(M)}\\
&\quad   - \tau ((\dt^4 \tilde{\phi}) u,u)_{L^2(M)}.
\end{align*}
Using the explicit form of $\tilde{\phi}$, we have $\p_t^4\tilde{\phi}(t)\leq (9/8) \p_t^2\tilde{\phi}(t)$ for all $t\leq 0$. Hence, the last term can be absorbed into the other contributions if  $\tau \geq 1$:
\begin{align*}
([S,A]u,u)_{L^2(M)}& \geq  4\tau \left\| (\dt^2 \tilde{\phi})^{1/2} \dt u \right\|_{L^2(M)}^2 +  2\tau^3 \left\| (\dt^2 \tilde{\phi})^{1/2}(\dt \tilde{\phi}) u \right\|_{L^2(M)}^2.
\end{align*}
Hence, we obtain
\begin{align*}
\left\| (S+A)u \right\|_{L^2(M)}^2 & \geq  \left\| S u \right\|_{L^2(M)}^2 + \left\| A u \right\|_{L^2(M)}^2 \\
& \quad + 4\tau \left\| (\dt^2 \tilde{\phi})^{1/2} \dt u \right\|_{L^2(M)}^2 +  2\tau^3 \left\| (\dt^2 \tilde{\phi})^{1/2}(\dt \tilde{\phi}) u \right\|_{L^2(M)}^2.
\end{align*}
Last but not least, we upgrade the gradient estimate from an estimate for the radial derivative to an estimate for the full gradient: Using the symmetric part of the operator, we deduce
\begin{equation}
\label{eq:gradup}
\begin{split}
&\left\| (\dt^2 \tilde{\phi})^{\frac{1}{2}} \nabla_{S^n} u \right\|_{L^2(M)}^2 + \left\| (\dt^2 \tilde{\phi})^{\frac{1}{2}} \dt u \right\|_{L^2(M)}^2 \\
&= -((\dt^2 \tilde{\phi}) u, \D_{S^n} u)_{L^2(M)} - ((\dt^2 \tilde{\phi}) u, \dt^2 u)_{L^2(M)} \\
& \quad + \int\limits_{\R \times S^{n-1}} (\dt^2 \tilde{\phi}) (\nu \cdot \nabla_{S^n} u ) u d\mathcal{H}^{n-1} dt\\
& = -((\dt^2 \tilde{\phi})u, S u)_{L^2(M)} + \tau^2 ((\dt^2 \tilde{\phi})(\dt \tilde{\phi})^2 u, u)_{L^2(M)} \\
& \quad + 
\frac{(n-1)^2}{4}((\dt^2 \tilde{\phi})u,u)_{L^2(M)}
\\
& \leq  \tau \left\| (\dt^2 \tilde{\phi}) u \right\|_{L^2(M)}^2 + \frac{1}{\tau} \left\| Su \right\|_{L^2(M)}^2 + \tau^2 \left\| (\dt^2 \tilde{\phi})^{\frac{1}{2}}(\dt \tilde{\phi}) u \right\|_{L^2(M)}^2 \\
& \quad + \frac{(n-1)^2}{4}\left\| (\dt^2 \tilde{\phi})^{\frac{1}{2}} u\right\|_{L^2(M)}^2. 
\end{split}
\end{equation}
Again we used the complementary boundary conditions to dispose of the boundary integral. Observing that all right hand side contributions in (\ref{eq:gradup}) can be absorbed into the left hand side of the Carleman inequality and that this remains true if we multiply (\ref{eq:gradup}) with the factor $c \tau$ for a small, positive constant $c$, we obtain the desired full gradient estimate.\\
Thus, inserting the changes we made in passing to our conformal change of variables, i.e. $w = e^{\frac{n-1}{2}t} e^{-\tau \phi}u$, and recalling the changes in the volume element, results in the claimed inequality \eqref{eq:vCarl00}.\\

\emph{Step 3: Bounds for the right hand side of (\ref{eq:perturb}).}
We now proceed to estimating the right hand side of (\ref{eq:perturb}) and (\ref{eq:vCarl}). Applying the triangle inequality immediately leads to
\begin{align*}
\left\|e^{\tau \phi}|x|\D w \right\|_{L^2(A_{\rho,r}^+)} \leq \left\| e^{\tau \phi}|x|f \right\|_{L^2(A_{\rho,r}^+)} + \left\| e^{\tau \phi} |x|\p_i (a^{ij}-\delta^{ij}) \p_j w \right\|_{L^2(A_{\rho,r}^+)}.
\end{align*}
Thus, it remains to control the second contribution. Setting $v= e^{\tau \phi}w$, in Cartesian coordinates after conjugation, we have 
\begin{align*}
&e^{\tau \phi} |x|\p_i (a^{ij}-\delta^{ij}) \p_j w = |x|[(\p_i a^{ij})\p_j v - \tau (\p_j \phi)(\p_i a^{ij}) v + (a^{ij}-\delta^{ij})\p_{ij}v \\
&\quad - 2\tau (a^{ij}-\delta^{ij})(\p_i \phi)\p_j v + \tau^2 (a^{ij}-\delta^{ij})(\p_i \phi)(\p_j \phi)v - \tau (a^{ij}-\delta^{ij})(\p_{ij} \phi)v ].
\end{align*}
Since $| a^{ij}(x) -\delta^{ij}|\leq \sup\limits_{A_{\rho,r}^+}\left| \frac{a^{ij}-\delta^{ij}}{|x|^{\gamma}} \right| |x|^{\gamma} \leq C(a^{ij}) |x| ^{\gamma}$, we can directly bound all the contributions which neither contain derivatives of $a^{ij}$ nor second order derivatives of $v$. More precisely, we have
\begin{align*}
&2 \tau \left\| |x| (a^{ij}-\delta^{ij})(\p_i \phi)\p_j v\right\|_{L^2} + \tau^2 \left\|  |x| (a^{ij}-\delta^{ij})(\p_i \phi)(\p_j \phi)v \right\|_{L^2} \\
&\quad + \tau \left\| |x|(a^{ij}-\delta^{ij})(\p_{i j} \phi) v  \right\|_{L^2} \\
& \lesssim \tau \left\|  (a^{ij}-\delta^{ij})\p_j v\right\|_{L^2} + \tau^2 \left\|  |x|^{-1} (a^{ij}-\delta^{ij})v \right\|_{L^2} \\
& \lesssim \tau C(a^{ij})\left\|  |x|^{\gamma} \p_j v\right\|_{L^2} + \tau^2 C(a^{ij})  \left\|  |x|^{\gamma -1} v \right\|_{L^2},
\end{align*}
where we used that $|\nabla \phi(x)| \leq C|x|^{-1}, \ |\nabla^2 \phi(x)| \leq C|x|^{-2}$. Thus, using that the gradient term can be controlled by the $L^2$ contribution on the right hand side of the last inequality (c.f. the comments at the beginning of the proof of Lemma \ref{lem:H2}), these errors are of the form of the contributions on the right hand side of (\ref{eq:vCarl}).\\

We continue with the contribution which involves second derivatives of $v$. We first decompose the $L^2$ norm over $B_1^+$ into $L^2$ norms on the dyadic (half) annuli $ A_m^{+}:=B_{2^{-m}}^+ \setminus B_{2^{-m-1}}^+$:
\begin{align*}
\left\| |x|(a^{ij}-\delta^{ij})\p_{ij}v \right\|_{L^2(B_1^+)}
&\leq \sum\limits_{m\in \N} 2^{-m} \sup\limits_{A_m^+}|a^{ij}(x)-\delta^{ij}| \left\| \p_{ij}v \right\|_{L^2(A_m^+)} .
\end{align*}
On each annulus we invoke Lemma \ref{lem:H2}:
\begin{align*}
\left\| \p_{ij}v \right\|_{L^2(A_m^+)} 
\lesssim \tau^2 2^{2m}\left\| v\right\|_{L^2(A_{m-1}^+\cup A_m^+ \cup A_{m+1}^+)} + \left\|f \right\|_{L^2(A_{m-1}^+\cup A_m^+ \cup A_{m+1}^+)}.
\end{align*} 
Using the finite number of overlaps of the annuli in the right hand side contributions, we can sum over these. Combining this with the observation that
\begin{align*}
 \sup\limits_{A_m^+}|a^{ij}(x)-\delta^{ij}| \leq 2\quad \mbox{ and }\quad  \sup\limits_{A_m^+}|a^{ij}(x)-\delta^{ij}| \leq C(a^{ij}) 2^{- m \gamma}, 
\end{align*}
then results in 
\begin{align*}
&\left\| |x|(a^{ij}-\delta^{ij}) \p_{ij}v \right\|_{L^2(B_1^+)}\\
& \lesssim \sum\limits_{m\in \N} 2^{-m} \sup\limits_{A_m^+}|a^{ij}(x)-\delta^{ij}|\left( \tau^2 2^{2 m} \left\| v\right\|_{L^2(A_{m-1}^+\cup A_m^+ \cup A_{m+1}^+)} \right.\\
&\quad \left. + \left\|f \right\|_{L^2(A_{m-1}^+\cup A_m^+ \cup A_{m+1}^+)}\right)\\
& \lesssim \sum\limits_{m\in \N} \left(C(a^{ij})\tau^2 2^{m(1-\gamma)} \left\| v\right\|_{L^2(A_{m-1}^+\cup A_m^+ \cup A_{m+1}^+)}\right.\\
& \quad + \left. 2^{-m}\left\|f \right\|_{L^2(A_{m-1}^+\cup A_m^+ \cup A_{m+1}^+)}\right)\\
& \lesssim C(a^{ij})\tau^2 \left\| |x|^{\gamma-1} v\right\|_{L^2(B_1^+)}+  \left\||x| f \right\|_{L^2(B_1^+)}.
\end{align*}
All of these are terms as in the right hand side of (\ref{eq:vCarl}).\\
Hence, it remains to consider the terms which carry derivatives on the coefficients $a^{ij}$.
Introducing a partition of unity and using Hölder's inequality, we have
\begin{align*}
\left\| |x|(\p_i a^{ij})\p_j v \right\|_{L^2(B_1^+)}
&\leq \sum\limits_{m\in \N} \left\| \p_i a^{ij} \right\|_{L^{n+1}(A_m^+)} \left\| |x| \eta_m \p_j v \right\|_{L^{\frac{2(n+1)}{n-1}}(B_1^+)}.
\end{align*}
Here we assume that each of the functions $\eta_m$ is supported in the annulus $A_{2^{-m-1},2^{-m+1}}$ and satisfies derivative bounds of the form $|\p_j \eta_m| \lesssim 2^{m}$.
Sobolev's inequality then yields 
\begin{equation}
\label{eq:p_a_contribution}
\begin{split}
\left\| |x| \eta_m \p_j v \right\|_{L^{\frac{2(n+1)}{n-1}}(B_1^+)} &\lesssim 2^{-m}\left\| \eta_m  \p_j v \right\|_{L^{\frac{2(n+1)}{n-1}}(B_1^+)}\\
&\lesssim \left\| |x|\p_{ij} v \right\|_{L^2(A_{m-1}^+\cup A_m^+ \cup A_{m+1}^+)}  \\
& \quad +  \left\|   \p_j v \right\|_{L^2(A_{m-1}^+\cup A_m^+ \cup A_{m+1}^+)}.
\end{split}
\end{equation}
Invoking the $H^2$ estimates from Lemma \ref{lem:H2} again, we obtain
\begin{align*}
\left\| |x| \p_{ij} v \right\|_{L^2(A_m^+)}
& \lesssim \tau^2 2^{m}\left\| v\right\|_{L^2(A_{m-1}^+\cup A_m^+ \cup A_{m+1}^+)} + 2^{-m }\left\|f \right\|_{L^2(A_{m-1}^+\cup A_m^+ \cup A_{m+1}^+)}.
\end{align*}
Using a combination of
\begin{align*}
\left\| \p_i a^{ij} \right\|_{L^{n+1}(A_m^+)}  \leq \delta \mbox{ and } \left\| \p_i a^{ij} \right\|_{L^{n+1}(A_m^+)} \leq 2^{-m}C(a^{ij}),
\end{align*}
we again sum these contributions over $m$:
\begin{align*}
\left\| |x|(\p_i a^{ij})\p_j v \right\|_{L^2(B_1^+)}
&\lesssim \sum\limits_{m\in \N} \left\| \p_i a^{ij} \right\|_{L^{n+1}(A_m^+)} \left\| |x| \eta_m \p_j v \right\|_{L^{\frac{2(n+1)}{n-1}}(B_1^+)}\\
& \lesssim \sum\limits_{m\in \N} \left\| \p_i a^{ij} \right\|_{L^{n+1}(A_m^+)} (\tau^2 2^{m}\left\| v\right\|_{L^2(A_{m-1}^+\cup A_m^+ \cup A_{m+1}^+)} \\
& \quad + 2^{-m }\left\|f \right\|_{L^2(A_{m-1}^+\cup A_m^+ \cup A_{m+1}^+)})\\
& \lesssim \sum\limits_{m\in \N} (C(a^{ij})\tau^2 2^{m(1-\gamma)}\left\| v\right\|_{L^2(A_{m-1}^+\cup A_m^+ \cup A_{m+1}^+)} \\
& \quad + \delta 2^{-m }\left\|f \right\|_{L^2(A_{m-1}^+\cup A_m^+ \cup A_{m+1}^+)})\\
& \lesssim C(a^{ij})\tau^2 \left\| |x|^{-1+\gamma} v\right\|_{L^2(B_1^+)}  + \delta \left\||x| f \right\|_{L^2(B_1^+)}.
\end{align*}
This can be absorbed into the right hand side of the Carleman inequality (\ref{eq:vCarl}).
For the $\tau |x|(\p_i a^{ij})(\p_j \phi) v$ contribution we argue analogously and include the resulting error contributions in the right hand side of the Carleman inequality.
\end{proof}

\begin{rmk}
\label{rmk:regularization}
Using regularizations $f_m \in L^p$ with $p>n+1$ of the inhomogeneity $f$ and regularizations $a^{ij}_{m}\in W^{1,p}$ with $p>n+1$ of the metric $a^{ij}$, such that
\begin{align*}
f_m \rightarrow f \in L^{n+1}(A_{\rho,r}^+) \mbox{ and } a^{i j}_m \rightarrow a^{ij} \mbox{ in } W^{1,n+1}(A_{\rho,r}^+) \mbox{ as } m \rightarrow \infty,
\end{align*}
would have allowed us to work with $C^1$ solutions, $w_m$, of the regularized problem (\ref{eq:problem}), since on the regularized level the regularity theory of Uraltseva \cite{U87} is available. Passing to the limit in the regularization parameter and observing that (\ref{eq:vCarl}) does not depend on the higher $L^p$ norms of the regularizations, would have provided another strategy of deriving the Carleman inequality (\ref{eq:vCarl}) and of interpreting the associated boundary conditions.
\end{rmk}

\begin{rmk}[Antisymmetric part]
\label{rmk:anti}
It is possible to strengthen the Carleman inequality (\ref{eq:vCarl}) by including the antisymmetric part into the left hand side.
As in \cite{Rue}, an additional bound can be deduced from the antisymmetric part of the Carleman estimate for functions which are supported in an annulus of the form $A_{r_1,r_2}^+$. In Cartesian coordinates this estimate reads
\begin{align*}
&\tau^2  \left\| |x|^{-1} e^{\tau \phi} w \right\|_{L^2(A_{r_1, r_2}^+) }^2 \lesssim  \ (\ln(r_2/r_1))^2\left\|e^{\tau \phi}  |x| \D w \right\|_{L^2(A_{r_1, r_2}^+) }^2.
\end{align*}
Thus, for functions which are supported in $A_{r_1, r_2}^+$, the Carleman estimate from Proposition \ref{prop:varmet} can be strengthened to
\begin{equation}
\label{eq:vCarl_a}
\begin{split}
&\tau^{\frac{3}{2}} \left\|  e^{\tau \phi}|x|^{-1} (1+\ln(|x|)^2)^{-\frac{1}{2}}  w \right\|_{L^2(A_{r_1, r_2}^+)} 
+  \tau^{\frac{1}{2}}\left\| e^{\tau \phi} (1+\ln(|x|)^2)^{-\frac{1}{2}}  \nabla w \right\|_{L^2(A_{r_1,r_2}^+)}\\
&+ \tau (\ln(r_2/r_1))^{-1}\left\| |x|^{-1} e^{\tau \phi} w \right\|_{L^2(A_{r_1, r_2}^+) }\\
&\leq c_0^{-1}C(n) \left(  
 \tau^2 C(a^{ij}) \left\| e^{\tau\phi}|x|^{\gamma-1} w \right\|_{L^2(A_{r_1,r}^+)} + \left\|e^{\tau \phi} |x| f\right\|_{L^2(A_{r_1,r_2}^+)} \right),
\end{split}
\end{equation}
for all $\tau>1$.
We will use this observation in Corollary \ref{cor:consequence_Carl}.
\end{rmk}

In Section \ref{sec:homo}, we will search for homogeneous blow-up solutions. The argument for the existence of these homogeneous blow-up solutions will rely on an improved Carleman estimate, which states that if a solution is far from homogeneous (c.f. (\ref{eq:nonhomo})), it is possible to work with non-convex weight functions (c.f. condition (b)). 

\begin{lem}
\label{lem:varmet1}
Let $\tau>1$, $n\geq 2$. Let $\rho, r$ be two radii with $0<\rho<r<1$.
Assume that $w\in H^1(B_1^+)$ with $\supp(w) \subset \overline{A^+_{\rho, r}}$ satisfies the following condition:
There exist $\mu>0$ and $\epsilon\in (0,1)$ such that for all $ j\in \{1,..., \floor{\ln( r/ \rho )}-2\}$
\begin{equation}
\label{eq:nonhomo}
\left\|(x\cdot \nabla - \mu ) w \right\|_{L^2(A_{e^j \rho, e^{j+1} \rho}^+)} \geq \epsilon \left\| w \right\|_{L^2(A_{e^{j} \rho, e^{j+1} \rho}^+)}.
\end{equation}
Moreover, suppose that it is a weak solution of the divergence form equation
\begin{align*}
\partial_i a^{ij}\partial_j w &=f \text{ in } A_{\rho, r}^+,\\
w \geq 0, \ \p_{n+1}w \leq 0, \ w \p_{n+1}w &=0 \text{ on } A'_{\rho, r},
\end{align*}
where $f:A_{\rho,r}^+ \rightarrow \R$ is an in $\overline{A_{\rho,r}^+}$ compactly supported $L^2(A_{\rho,r}^+)$ function and $a^{ij}: A_{\rho,r}^+ \rightarrow \R^{n\times n}_{sym}$ is a tensor field such that
\begin{itemize}
\item[(i)] $a^{ij}\in W^{1,n+1}(A_{\rho,r}^+)$, 
\item[(ii)] the off-diagonal assumption (A1),
\item[(iii)] the uniform ellipticity assumption (A2),
\item[(iv)] and the following smallness condition hold: There exists $\delta = \delta(n, \epsilon, \mu)>0$ such that 
\begin{equation}
\label{eq:small1}
\begin{split}
\tau^2\sup\limits_{\rho \leq \bar{r} \leq r} \left\| \nabla a^{ij} \right\|_{L^{n+1}(A_{\bar{r}, e \bar{r}}^+)} &\leq \delta \mbox{ and }\\
\tau^2\sup\limits_{\rho \leq \bar{r} \leq r} |a^{ij}(x)-\delta^{ij}|&\leq \delta.
\end{split}
\end{equation}
\end{itemize}
Let $\tilde{\phi}:\R\rightarrow \R$, $\tilde{\phi}\in C^{4}$, be a radial weight function satisfying the following assumptions:
\begin{itemize}
\item[(a)] $|\tau \tilde{\phi}'(t) +\mu + \frac{n-1}{2}| \leq \frac{\epsilon}{2}$,
\item[(b)] $\tau\tilde{\phi}''(t) \geq  -\frac{\tilde{\epsilon}}{10 }$,
\item[(c)] $  \tau|\tilde{\phi}^{''''}(t)|\leq \frac{\tilde{\epsilon}}{10}$,
\end{itemize}
where $\tilde{\epsilon}:= \frac{\epsilon}{16} e^{-\left( n + \mu \right)}$.
Set $\phi(x)= \tilde{\phi}(\ln(|x|))$.
Then we have
\begin{equation}
\label{eq:vCarl1}
\begin{split}
& \left\|  e^{\tau \phi}|x|^{-1} w \right\|_{L^2(A_{\rho, r}^+)} 
+  \left\| e^{\tau \phi}  \nabla w \right\|_{L^2(A_{\rho,r}^+)} \leq C(n, \epsilon, \mu) \left\|e^{\tau \phi} |x| f\right\|_{L^2(A_{\rho,r}^+)}.
\end{split}
\end{equation} 
\end{lem}

Here assumption (\ref{eq:nonhomo}) allows for slightly concave weight functions satisfying the hypotheses (a)-(c). While this permits us to strengthen the Carleman estimate by using (slightly) non-concave weight functions $\tilde{\phi}$, we have slight losses in the powers of $\tau$ that appear as (polynomial) factors in the estimate (\ref{eq:vCarl1}).

\begin{rmk}
\begin{itemize}
\item As before, estimate (\ref{eq:vCarl1}) remains valid for $n=1$ with only slight modifications.
\item Similarly as in Remark \ref{rmk:tau}, the Carleman estimate remains valid for any $\tau >0$.
\end{itemize}
\end{rmk}

\begin{proof}
We argue along the lines of the proof of the Carleman estimate of Proposition \ref{prop:varmet}. However, in the setting of Lemma \ref{lem:varmet1} assumption (\ref{eq:nonhomo}) allows us to bound the commutator contributions, in spite of a potential (slight) concavity of the weight $\tilde{\phi}$. To this end we have to estimate the commutator contributions
\begin{equation}
\label{eq:nonconvex_error}
4\tau ((\dt^2 \tilde{\phi}) \dt u, \dt u)_{L^2(M)} + 4 \tau^3 ((\dt^2 \tilde{\phi})(\dt \tilde{\phi})^2 u, u)_{L^2(M)}  - \tau ((\dt^4 \tilde{\phi}) u,u)_{L^2(M)}.
\end{equation}
Due to the combination of assumptions (a)-(c) and the ``non-homogeneity condition'' (\ref{eq:nonhomo}), it is possible to absorb these contributions into the antisymmetric part of the operator, whenever the weight $\tilde{\phi}$ becomes (slightly) concave. Indeed, after the changes into conformal coordinates and after setting $u=e^{\frac{n-1}{2}t}e^{\tau \tilde{\phi}} w$ as well as $t:=\ln(r)$, $L_0:=\floor{\ln(r/\rho)}$, (\ref{eq:nonhomo}) turns into
\begin{equation*}
\begin{split}
&\left\| e^{-\tau \phi - \frac{n-1}{2}t}\left(\dt u - \mu u - \frac{n-1}{2}u - \tau \dt \tilde{\phi}  \right) \right\|_{L^2((t+j,t+j+1)\times S^{n}_+)}\\
& \geq \epsilon \left\| e^{-\tau \phi - \frac{n-1}{2}t} u\right\|_{L^2((t+j,t+j+1)\times S^{n}_+)}, \quad  j=1,\dots, L_0-2.
\end{split}
\end{equation*}
Thus, for $j=1,\dots, L_0-2$,
\begin{align*}
\left\| e^{-\tau \phi - \frac{n-1}{2}t}\dt u\right\|_{L^2((t+j,t+j+1)\times S^{n}_+)}\geq \epsilon \left\| e^{-\tau \phi - \frac{n-1}{2}t} u\right\|_{L^2((t+j,t+j+1)\times S^{n}_+)}\\-\left\| e^{-\tau \phi - \frac{n-1}{2}t}\left(\mu  + \frac{n-1}{2} + \tau \dt \tilde{\phi}  \right) u\right\|_{L^2((t+j,t+j+1)\times S^{n}_+)}.
\end{align*}
Recalling our assumption (a), we hence obtain
\begin{equation}
\label{eq:nonhomo1}
\left\| e^{-\tau \phi - \frac{n-1}{2}t}\dt u\right\|_{L^2((t+j,t+j+1)\times S^{n}_+)}\geq \frac{\epsilon}{2} \left\| e^{-\tau \phi - \frac{n-1}{2}t} u\right\|_{L^2((t+j,t+j+1)\times S^{n}_+)}.
\end{equation}
Invoking the fundamental theorem of calculus to estimate 
\begin{align*}
|\tau \phi(t_1)- \tau \phi(t_2)| \leq \int\limits_{t+j}^{t+j+1}|\tau \tilde{\phi}'(s)|ds \leq \left( \mu + \frac{n-1}{2} + \frac{\epsilon}{2} \right),
\end{align*}
for all $t_1,t_2 \in (t+j,t+j+1)$, permits us to further bound the quantities in (\ref{eq:nonhomo1}): For $j=1,\dots, L_0-2$
\begin{equation}
\label{eq:nonhomo2}
\begin{split}
\left\| \dt u \right\|_{L^2((t+j,t+j+1)\times S^{n}_+)} &\geq \frac{\epsilon}{2}e^{-\left( \mu+ n-1  +\frac{\epsilon}{2}\right)} \left\|  u\right\|_{L^2((t+j,t+j+1)\times S^{n}_+)}\\
&\geq 8 \tilde{\epsilon} \left\|  u\right\|_{L^2((t+j,t+j+1)\times S^{n}_+)} .
\end{split}
\end{equation}
At the boundary annuli (which correspond to $j=0$ or $j=L_0-1$) we do not have the ``non-homogeneity assumption'' (\ref{eq:nonhomo}) at our disposal. Yet Poincar\'e's inequality and the compact support of $u$ imply
\begin{equation}
\label{eq:ann_bound}
\begin{split}
\|u\|_{L^2(t,t+1)}&\lesssim \|\partial_t u\|_{L^2(t,t+1)},\\
\|u\|_{L^2(t+L_0-1, t+L_0+1)}&\lesssim \|\partial_t u\|_{L^2(t+L_0-1, t+L_0+1)}.
\end{split}
\end{equation}
Using (\ref{eq:nonhomo2}) as well as (\ref{eq:ann_bound}), we can proceed to control the commutator contributions in (\ref{eq:nonconvex_error}) from the Carleman inequality. In the regions in which $\tilde{\phi}$ is (locally) convex, we argue as in the proof of Proposition \ref{prop:varmet}. Hence, it suffices to consider the regions in which $\tilde{\phi}$ is not convex. Let $U\subset M$ be such a region. By assumption (b), we have
\begin{equation}
\label{eq:antisymm1}
\begin{split}
\left\|  A u \right\|_{L^2(U)} &= \tau \left\| \dt \tilde{\phi} \dt u - \dt^2 \tilde{\phi} u \right\|_{L^2(U)} \geq  \tau \left\| \dt \tilde{\phi} \dt u \right\|_{L^2(U)} - \tau \left\| \dt^2 \tilde{\phi} u \right\|_{L^2(U)}\\
& \geq \left( \mu + \frac{n-1}{2} - \frac{\epsilon}{2} \right) \left\| \dt u \right\|_{L^2(U)} -  \frac{\tilde{\epsilon}}{10}\left\| u\right\|_{L^2(U)}.
\end{split}
\end{equation}
The gradient contributions in (\ref{eq:nonconvex_error}), are immediately controlled by combining the estimates (\ref{eq:antisymm1}) and (\ref{eq:nonhomo2}) (or respectively (\ref{eq:antisymm1}) and (\ref{eq:ann_bound}) if $U$ contains the boundary annuli). For the $L^2$ contributions in (\ref{eq:nonconvex_error}) we use (\ref{eq:nonhomo2}), (\ref{eq:antisymm1}) and the conditions (a)-(c) (or an analogous argument in which (\ref{eq:nonhomo2}) is replaced by (\ref{eq:ann_bound})). After having absorbed the error contributions from (\ref{eq:nonconvex_error}) into $\frac{1}{2}\left\| A u \right\|_{L^2(M)} $, the remaining $\frac{1}{2}\left\| A u \right\|_{L^2(M)} $ contribution then yields an $L^2$ lower bound (by virtue of (\ref{eq:nonhomo2}), (\ref{eq:ann_bound}) and (\ref{eq:antisymm1})). Moreover, an estimate similar as in (\ref{eq:gradup}) (without the $\dt^2 \tilde{\phi}$ weight) allows us to complement the $L^2$ contributions by gradient contributions. Thus, we arrive at an analogue of equation (\ref{eq:vCarl00}):
\begin{align*}
& \left\|  e^{\tau \phi}|x|^{-1} w \right\|_{L^2(A_{\rho, r}^+)} 
+  \left\| e^{\tau \phi}  \nabla w \right\|_{L^2(A_{\rho,r}^+)} \leq C(n,\epsilon,\mu) \left\| e^{\tau \phi}|x|\Delta w \right\|_{L^2(A_{\rho,r}^+)}.
\end{align*}
We stress that in contrast to (\ref{eq:vCarl00}), we now do not face logarithmic losses since we have (\ref{eq:nonhomo}) at our disposal.\\
After this observation, the remainder of the proof of Lemma \ref{lem:varmet1} essentially follows along the same lines as in the proof of the Carleman estimate from Proposition \ref{prop:varmet}. We only consider small modifications which allow us to absorb the error contributions on the right hand side of (\ref{eq:vCarl1}) into the left hand side immediately. 
We briefly comment on these modifications. Instead of estimating the differences $|a^{ij}-\delta^{ij}|$ in terms of $C(a^{ij})$ as in step 3 of the proof of Proposition \ref{prop:varmet}, we  use the two smallness conditions from (\ref{eq:small1}). This becomes possible as we do not have to gain a factor $|x|^{\gamma}$, since we do not have logarithmic losses on the left hand side of (\ref{eq:vCarl1}).
\end{proof}

We will use the modified Carleman estimate from Lemma \ref{lem:varmet1} in proving Lemma \ref{lem:almosthom} in Section \ref{sec:homo}. There we deduce a condition which ensures closeness to homogeneous functions (along sequences of certain radii).

\subsection{Consequences of the Carleman estimate (\ref{eq:vCarl})}
We continue by deriving a corollary from the Carleman estimate. For us the corollary will play a fundamental role in the following section. In particular it will be used to deduce bounds on the vanishing order (c.f. Proposition \ref{prop:indep}, Corollary \ref{cor:growth1} and Lemma \ref{lem:growthimp}), the upper semi-continuity of the free boundary (c.f. Proposition \ref{prop:semi_cont}) and doubling properties for solutions of (\ref{eq:varcoef}) (c.f. Proposition \ref{prop:doubling}).

\begin{cor}[Three spheres inequality]
\label{cor:consequence_Carl}
Let $\tilde{\phi}$ be the weight function from Proposition \ref{prop:varmet}.
Suppose that $w:B_{1}^+ \rightarrow \R$ is a solution of the thin obstacle problem (\ref{eq:varcoef}) satisfying (A0)-(A4) and that $x_0 \in B_{1/2}^+$. Assume that $1\leq \tau<\tau_0 < \infty$. Assume that $R_0= R_0(\tau_0,n,p,\|\nabla a^{ij}\|_{L^p})\in(0,\frac{1}{2})$ is the radius from Remark \ref{rmk:Carl_application} and that $0<r_1< r_2<r_3\leq R_0 $ and $r_1 \leq \frac{2}{3}r_3$. Then there is a constant $C=C(n)$ such that
\begin{equation}
\label{eq:consequence_Carl1}
\begin{split}
&\tau^{\frac{3}{2}}(1+|\ln(r_2)|)^{-1}  e^{\tau \tilde{\phi}(\ln(r_2))} r_2^{-1}\left\| w \right\|_{L^2(A^+_{r_2,2r_2}(x_0))}\\
&+ \tau \max\{\ln(r_2/r_1)^{-1},\ln(r_3/r_2)^{-1}\} e^{\tau \tilde{\phi}(\ln(r_2))}r_2^{-1}  \left\| w \right\|_{L^2(A^+_{r_2,2r_2}(x_0))} \\
& \leq c_0^{-1}C\left( e^{\tau \tilde{\phi}(\ln(r_1))} r_1^{-1} \left\| w \right\|_{L^2(A^+_{r_1,2r_1}(x_0))}  + e^{\tau \tilde{\phi}(\ln(r_3))}r_3^{-1}\left\| w \right\|_{L^2(A^+_{r_3,2 r_{3}}(x_0))}\right).
\end{split}
\end{equation}
\end{cor}

\begin{rmk}
We stress that the constant $C>0$ in Corollary \ref{cor:consequence_Carl} is uniform in the points $x_0\in B_{\frac{1}{2}}^+$.
\end{rmk}

\begin{rmk}
If additionally to the assumptions of Corollary \ref{cor:consequence_Carl}, the estimate \eqref{eq:nonhomo} from Lemma \ref{lem:varmet1} is satisfied, then a similar inequality as in \eqref{eq:consequence_Carl1} holds true for the weight function $\tilde{\phi}$ from Lemma~\ref{lem:varmet1}. More precisely, let $\tau$ and $\tilde{\phi}$ be in Lemma~\ref{lem:varmet1}. Let $R_0=R_0(\mu, \epsilon, \tau, n,p, \|\nabla a^{ij}\|_{L^p})$ be such that \eqref{eq:small1} is satisfied for all $\rho, r\in (0,R_0)$. Then for $0<2 r_1< r_2< \frac{r_3}{2}\leq R_0$,
\begin{equation}
\label{eq:consequence_Carl2}
\begin{split}
&e^{\tau \tilde{\phi}(\ln(r_2))} r_2^{-1}\left\| w \right\|_{L^2(A^+_{r_2,2r_2}(x_0))}\\
& \leq C(n,\mu, \epsilon)\left( e^{\tau \tilde{\phi}(\ln(r_1))} r_1^{-1} \left\| w \right\|_{L^2(A^+_{r_1,2r_1}(x_0))}  + e^{\tau \tilde{\phi}(\ln(r_3))}r_3^{-1}\left\| w \right\|_{L^2(A^+_{r_3,2 r_{3}}(x_0))}\right).
\end{split}
\end{equation}
\end{rmk}

\begin{proof}
We only prove the statements for $x_0 =0$ and consider the Carleman weights, $\phi$ and $\tilde{\phi}$, from Proposition \ref{prop:varmet}.
Let $\eta$ be a cut-off function supported in $A_{5r_1/4,3r_3/2}$ with an only radial dependence such that
\begin{equation}
\label{eq:etab}
\begin{split}
\eta=1 \text{ in } A_{3 r_1/2, 5 r_3/4}, \quad \eta=0 \text{ in } \R^{n+1}\setminus A_{5 r_1/4, 3r_3/2},\\
|\nabla \eta |\leq \frac{2}{r_1} \text{ in } A_{5r_1/4, 3r_1/2}, \quad |\nabla \eta |\leq \frac{2}{r_3} \text{ in } A_{5r_3/4, 3r_3/2}.
\end{split}
\end{equation}
Then, as a consequence of the radial dependence of $\eta$ and the support assumptions on it, $\tilde{w}=w\eta$ is supported in $A_{5r_1/4,3r_3/2}$ and satisfies 
\begin{align*}
\partial_i(a^{ij}\partial_j \tilde{w})&=\partial_i(a^{ij}\partial_j (\eta w)) \\
&= 2 a^{ij}(\p_j \eta) \p_i w + a^{ij}(\p_{ij}\eta)w + (\p_{i}a^{ij})(\p_j \eta)w \text{ in } A^+_{r_1,2r_3},\\
\partial_{n+1} \tilde{w}&= \eta \p_{n+1}w = 0 \mbox{ on } A'_{r_1,2r_3}\setminus \Lambda_w.
\end{align*}
We notice that since $a^{ij}\in W^{1,p}$ and since $R_0$ is chosen as in Remark \ref{rmk:Carl_application}, the assumptions of the Carleman estimate (\ref{eq:vCarl}) are satisfied. 
Inserting $\tilde{w}$ into the Carleman estimate and by using the antisymmetric part of the operator combined with Poincar\'e's inequality (c.f. Remark \ref{rmk:anti}), i.e. respectively estimating 
\begin{align*}
\tau \||x|^{-1} v\|_{L^2(A_{2 r_1, r_2}^+)}\lesssim (\ln(r_2/r_1))\|Av\|_{L^2(A_{r_1, r_2}^+)},\\
\tau \| |x|^{-1} v\|_{L^2(A_{2 r_2, r_3}^+)}\lesssim (\ln(r_3/r_2))\|Av\|_{L^2(A_{r_2, r_3}^+)},\\
 \mbox{ where } v=e^{\tau\phi}\tilde{w},
\end{align*} 
leads to
\begin{equation}
\label{eq:cutoffcar}
\begin{split}
& \tau (\ln(r_2/r_1))^{-1} \left\| e^{\tau \phi} |x|^{-1} (w\eta)\right\|_{L^2(A^+_{r_1, r_2})}\\
&+  \tau (\ln(r_3/r_2))^{-1} \left\| e^{\tau \phi} |x|^{-1} (w\eta)\right\|_{L^2(A^+_{r_2, r_3})}\\
&+\tau^{\frac{3}{2}} \left\|  e^{\tau \phi}|x|^{-1} (1+\ln(|x|)^2)^{-\frac{1}{2}}  (w\eta) \right\|_{L^2(A^+_{r_1, 2r_3})} \\
& +\tau^{\frac{1}{2}}\left\| e^{\tau \phi} (1+\ln(|x|)^2)^{-\frac{1}{2}}  \nabla (w\eta) \right\|_{L^2(A^+_{r_1,2r_3})}\\
&\leq c_0^{-1}C(n)\left(   C(a^{ij})\tau^2\left\|e^{\tau \phi}|x|^{\gamma-1}  (w\eta) \right\|_{L^2(A^+_{r_1,2r_3})}\right.\\
& \quad+\left\|e^{\tau \phi} |x| a^{ij}(\partial_{ij}^2 \eta) w\right\|_{L^2(A^+_{r_1,2r_3})} + \left\|e^{\tau \phi} |x| a^{ij}(\partial_j\eta)( \p_i  w)\right\|_{L^2(A^+_{r_1,2 r_3})} \\
& \left.\quad +\left\|e^{\tau \phi} |x| (\p_i a^{ij})(\p_i \eta) w\right\|_{L^2(A^+_{r_1,2r_3})}\right),
\end{split}
\end{equation}
where, in the notation of Proposition \ref{prop:varmet}, $\gamma := 1-\frac{n+1}{p}$ is the Hölder differentiability modulus of $a^{ij}$. Using the fact that $0<\tau\leq \tau_0$ is bounded, we may immediately absorb the first contribution on the right hand side by choosing $R_0=R_0(\tau_0, \gamma,n, \left\| \nabla a^{ij} \right\|_{L^p})>0$ sufficiently small as in Remark~\ref{rmk:Carl_application}.
Using the properties (\ref{eq:etab}) of $\eta$ and the monotonicity of $\tilde{\phi}$, we infer
\begin{equation}
\label{eq:err1}
\begin{split}
&\left\|e^{\tau \phi} |x| a^{ij}(\partial_{i j}^2 \eta) w\right\|_{L^2(A^+_{r_1,r_3})} + \left\|e^{\tau \phi} |x| a^{ij}(\partial_j\eta)( \p_i  w)\right\|_{L^2(A^+_{r_1,r_3})} \\
&\lesssim\  \frac{1}{r_1} \left\|e^{\tau \phi} w\right\|_{L^2(A_{5r_1/4, 3r_1/2}^+)}+\frac{1}{r_3} \left\|e^{\tau \phi} w\right\|_{L^2(A_{5r_3/4, 3r_3/2}^+)}\\
&\quad +  \left\|e^{\tau \phi} \nabla w\right\|_{L^2(A_{5r_1/4, 3 r_1/2}^+)}+ \left\|e^{\tau \phi} \nabla w\right\|_{L^2(A_{5 r_3/4, 3 r_3/2}^+)}\\
& \lesssim\  \frac{1}{r_1} e^{\tau \tilde{\phi}(\ln(r_1))} \left\| w\right\|_{L^2(A_{5r_1/4, 3r_1/2}^+)}+\frac{1}{r_3} e^{\tau \tilde{\phi}(\ln(r_3))} \left\| w\right\|_{L^2(A_{5r_3/4, 3r_3/2}^+)}\\
&\quad +  e^{\tau \tilde{\phi}(\ln(r_1))} \left\| \nabla w\right\|_{L^2(A_{5r_1/4, 3 r_1/2}^+)}+ e^{\tau \tilde{\phi}(\ln(r_3))} \left\| \nabla w\right\|_{L^2(A_{5 r_3/4, 3 r_3/2}^+)}.
\end{split}
\end{equation}
In order to control the last term in (\ref{eq:cutoffcar}), we use the support condition on $\eta$ to pull out the exponential factor $e^{\tau \phi}$ and to estimate the remainder via a combination of Hölder's inequality and Sobolev's inequality:
\begin{equation}
\label{eq:err2}
\begin{split}
&\left\| |x|(\p_i a^{ij})(\p_j \eta)w \right\|_{L^2(A^+_{r_1, 2r_3})} \leq \left\| \p_i a^{ij}\right\|_{L^{n+1}(A^+_{r_1, 2r_3})}\left\|  w|x|(\p_j \eta) \right\|_{L^{\frac{2(n+1)}{n-1}}(A^+_{r_1, 2r_3})}\\
& \leq C(n) \left\| \p_i a^{ij}\right\|_{L^{n+1}(A^+_{r_1, 2r_3})}\left( \left\| (\nabla^2 \eta) |x| w \right\|_{L^2(A^+_{r_1, 2r_3})} + \left\| w \nabla \eta  \right\|_{L^2(A^+_{r_1, 2r_3})} \right. \\ \\
& \quad \left.+ \left\| |x| \nabla w \cdot \nabla \eta \right\|_{L^2(A^+_{r_1, 2r_3})}\right).
\end{split}
\end{equation}
Due to the support assumptions on $\eta$ and the choice of $R_0$, this results in estimates of the type (\ref{eq:err1}).
Pulling out the exponential factors in the remaining contributions, recalling the explicit form of the Carleman weight and combining (\ref{eq:err1}) and (\ref{eq:err2}) with Caccioppoli's inequality (for which we use the boundary conditions for $w$ and $\eta$ on $B_1'$), e.g.
\begin{align*}
\left\| \nabla w \right\|_{L^2(A^+_{5r_1/4,3r_1/2})} \lesssim r_1^{-1}\left\| w\right\|_{L^2(A^+_{r_1,2r_1})},
\end{align*}
results in
\begin{equation}
\label{eq:cutoffcar1}
\begin{split}
&\tau^{\frac{3}{2}}(1+|\ln(r_2)|)^{-1}  e^{\tau \tilde{\phi}(\ln(r_2))}\left\| w \right\|_{L^2(A^+_{r_2,2r_2})}\\
&+ \tau \max\{\ln(r_2/r_1)^{-1},\ln(r_3/r_2)^{-1}\} e^{\tau \tilde{\phi}(\ln(r_2))}r_2^{-1}  \left\| w \right\|_{L^2(A^+_{r_2,2r_2})}\\
&\leq c_0^{-1}C(n)\left(   e^{\tau \tilde{\phi}(\ln(r_1))} r_1^{-1}  \left\| w\right\|_{L^2(A_{r_1, 2r_1}^+)}+ e^{\tau \tilde{\phi}(\ln(r_3))}r_3^{-1}\left\| w\right\|_{L^2(A_{r_3, 2r_3}^+)}\right).
\end{split}
\end{equation}
This is the desired estimate.
\end{proof}

\subsection{Proof of Lemma \ref{lem:H2}}
\label{sec:H2}

We conclude this section with the proof of Lemma \ref{lem:H2}. Here we closely follow the ideas of Uraltseva \cite{U87}, however carefully keeping track of the powers of $\tau$ which are involved.

\begin{proof}[Proof of Lemma \ref{lem:H2}]
Let $\zeta$ be a smooth radial function supported in $B_1(0)\backslash B_{1/8}(0)$, idendically $1$ on $B_{1/2}(0)\backslash B_{1/4}(0)$. We multiply (\ref{eq:H2_est}) by 
$\zeta^2(2^{m-1}x)  v $ and integrate to infer 
\[ \Vert |x|^{-1} \nabla (\zeta(2^{m-1}.)  v) \Vert_{L^2} \lesssim  \tau \Vert |x|^{-2} v \Vert_{L^2(B_{2^{-m+1}}\backslash B_{2^{-m-2}})} +
\Vert  \zeta^2 f \Vert_{L^2} .  \] 
Then standard arguments show that we obtain the desired estimate (\ref{eq:H2}) if we prove it with an additional gradient term on the right hand side of (\ref{eq:H2}).\\

\emph{Step 1: Energy estimates.}
We consider the following penalized problem
\begin{equation}
\label{eq:penalize}
\begin{split}
 \p_i a^{ij}_{\epsilon} \p_j\ve &= g^{\epsilon} \mbox{ in } B_{1}^+,\\
\p_{n+1}\ve &= \beta_{\epsilon}(\ve) \mbox{ on } B_{1}',\\
\ve & = 0 \mbox{ on } \partial B_{1}^+.
\end{split}
\end{equation}
Here $\beta_{\epsilon}\in C^{\infty}(\R)$ is chosen such that
\begin{align*}
\beta_{\epsilon}\leq 0 , \ \beta_{\epsilon}' \geq 0, \ \beta_{\epsilon}(s)= 0 \mbox{ for } s\geq 0, \ \beta_{\epsilon}(s)= \epsilon + \frac{s}{\epsilon} \mbox{ for } s\leq -2\epsilon^2.
\end{align*} 
The function $a^{ij}_{\epsilon}$ is a regularization of $a^{ij}$ and $g^{\epsilon}$ is a regularization of
\begin{equation*}
g:=  f - \tau^2(\p_i \phi) a^{ij}(\p_j \phi) v + \tau a^{ij} (\p_i \phi) \p_j v + \tau (\p_j \phi) \p_i (a^{ij} v) + \tau a^{ij}(\p_{ij}^2 \phi) v.
\end{equation*}
As $a^{ij}_{\epsilon}$ is uniformly elliptic, a solution $\ve$  of (\ref{eq:penalize}) exists and is unique and smooth. Moreover, $\ve$ converges weakly in $H^1$ to a weak solution, $u$, of
\begin{equation}
\label{eq:penalize_limit}
\begin{split}
 \p_i a^{ij} \p_j u &= g \mbox{ in } B_{1}^+,\\
\p_{n+1}u &\leq  0, \ u \geq 0, \ u(\p_{n+1}u)=0 \mbox{ on } B_{1}',\\
u & = 0 \mbox{ on } \partial B_{1}^+.
\end{split}
\end{equation}
Here the boundary conditions in (\ref{eq:penalize_limit}) are interpreted distributionally, i.e. similarly as in Step 1 of the proof of Proposition \ref{prop:varmet}.
As the bilinear form associated with (\ref{eq:penalize_limit}) is coercive, the problem (\ref{eq:penalize_limit}) has a unique weak solution. Since $v$ satisfies (\ref{eq:penalize_limit}), this implies that $v=u$, so in particular, $\ve \rightharpoonup v$ in $H^1(B_1^+)$. \\

We now proceed to the $H^2$ estimate for $v^{\epsilon}$. Multiplying the equation with a compactly supported test function $\eta$, we start with the identity
\begin{align*}
\int\limits_{B_{1}^+} a^{ij}\p_j \ve \p_i \eta dx &= - \int\limits_{B_1'}\beta_{\epsilon}(\ve) \eta d\mathcal{H}^{n-1}  - \int\limits_{B_{1}^+} g^{\epsilon} \eta dx, 
\end{align*}
where, as above, $g^{\epsilon}$ is a regularization of $g$:
\begin{align*}
g= f-\tau^2 (\p_{i} \phi) a^{ij} (\p_j \phi)v +\tau a^{ij}(\p_i \phi) \p_j v +\tau (\p_j \phi)\p_i(a^{ij}v) +\tau a^{ij} (\p_{ij}^2 \phi) v. 
\end{align*}
We will prove 
\begin{equation*}
\begin{split}  
\Vert \nabla^2 \ve \Vert_{L^2(B_{2^{-m}}^+ \backslash B_{2^{-m-1}}^+)} 
&\le c \tau^2 \Vert |x|^{-2} \ve \Vert_{L^2 (B_{2^{1-m}}^+ \backslash B_{2^{-m-2}}^+)}\\
& \quad + c \Vert g^\epsilon  \Vert_{L^2 (B_{2^{1-m}}^+ \backslash B_{2^{-m-2}}^+)}.
\end{split}
\end{equation*}

Choosing $\eta = \partial_{\mu}\tilde{\eta}$, with $\mu \in \{1,...,n\}$ being a tangential direction, and integrating by parts, results in
\begin{align*}
\int\limits_{B_{1}^+} a^{ij}\p_{j\mu} \ve \p_i \tilde{\eta} dx + \int\limits_{B_{1}^+} (\p_{\mu} a^{ij})\p_{j\mu} \ve \p_i \tilde{\eta} dx &= - \int\limits_{B_1'}\beta_{\epsilon}'(\ve)\p_{\mu}\ve \tilde{\eta} d\mathcal{H}^{n-1}\\
& \quad  + \int\limits_{B_{1}^+} g^{\epsilon}\p_{\mu}\tilde{\eta} dx.
\end{align*}
We then set $\tilde{\eta}:= \zeta^2 \p_{\mu}\ve$ where $\zeta$ is a smooth cut-off function which is essentially localized in $A_m$ for some arbitrary but fixed value of $m\in \N$:
\begin{align*}
&\int\limits_{B_1^+}a^{ij}\p_{j\mu}\ve \p_{i\mu}\ve \zeta^2 dx + 2\int\limits_{B_{1}^+} a^{ij}\p_{j \mu}\ve \p_{\mu}\ve \zeta \p_{i} \zeta dx 
+ 2 \int\limits_{B_{1}^+}(\p_{\mu}a^{ij})\p_j \ve \p_{ \mu}\ve \zeta \p_i \zeta  dx \\
&+ \int\limits_{B_{1}^+}(\p_{\mu}a^{ij}) \p_{j}\ve \p_{i \mu}\ve \zeta^2 dx = - \int\limits_{B_{1}'} \beta_{\epsilon}'(\ve)(\p_{\mu}\ve)^2 \zeta^2 d\mathcal{H}^{n-1} + \int\limits_{B_{1}^+} g^{\epsilon}\p_{\mu}\tilde{\eta} dx\\
&\leq  \int\limits_{B_{1}^+} g^{\epsilon}\p_{\mu}\tilde{\eta} dx,
\end{align*}
for which we noticed that
\begin{align*}
- \int\limits_{B_{1}'} \beta_{\epsilon}'(\ve)(\p_{\mu}\ve)^2 \zeta^2 d\mathcal{H}^{n-1} \leq 0.
\end{align*}
Using the uniform ellipticity of $a^{ij}$ and Young's inequality, we thus obtain
\begin{align*}
&\int\limits_{B_1^+}|\p_{\mu i} \ve|^2 \zeta^2 dx \leq 
-2 \int\limits_{B_{1}^+} a^{ij} \p_{\mu j}\ve \p_{\mu}\ve \zeta \p_{i}\zeta dx 
- 2\int\limits_{B_{1}^+}(\p_{\mu}a^{ij})\p_{j}\ve \p_{\mu}\ve \zeta \p_i \zeta dx \\
&- \int\limits_{B_{1}^+}(\p_{\mu}a^{ij}) \p_j \ve \p_{i\mu}\ve \zeta^2 dx + \int\limits_{B_{1}^+} g^{\epsilon}\p_{\mu}\tilde{\eta} dx\\
&\leq \delta \int\limits_{B_{1}^+}|\p_{j\mu}\ve|^2 \zeta^2 dx + C(\delta)\int\limits_{B_1^+}|\p_{\mu}\ve|^2|\p_{i}\zeta|^2 dx \\
&+ 2\int\limits_{B_{1}^+}|\p_{\mu}a^{ij}||\p_j \ve||\p_{\mu}\ve||\zeta||\p_i \zeta|dx 
+ \int\limits_{B_{1}^+}|\p_{\mu}a^{ij}||\p_j \ve||\p_{i\mu}\ve|\zeta^2 dx \\
&+ \int\limits_{B_{1}^+} g^{\epsilon}\p_{\mu}\tilde{\eta} dx .
\end{align*}
Recalling the equation for $\ve$, this control of mixed tangential-tangential and tangential-normal second order derivatives can be upgraded to a full second order derivative control at the expense of additional error contributions: 
\begin{equation}
\label{eq:2der}
\begin{split}
&\int\limits_{B_1^+}|\nabla^2 \ve|^2 \zeta^2 dx
\leq \delta \int\limits_{B_{1}^+}|\p_{j\mu}\ve|^2 \zeta^2 dx + C(\delta)\int\limits_{B_1^+}|\p_{\mu}\ve|^2|\p_{i}\zeta|^2 dx \\
&+ 2\int\limits_{B_{1}^+}|\p_{\mu}a^{ij}||\p_j \ve||\p_{\mu}\ve||\zeta||\p_i \zeta|dx 
+ \int\limits_{B_{1}^+}|\p_{\mu}a^{ij}||\p_j \ve||\p_{i\mu}\ve|\zeta^2 dx \\
&+ C \int\limits_{B_{1}^+}\zeta^2|\p_i a^{ij}|^2|\p_{j}\ve|^2dx + \int\limits_{B_{1}^+} g^{\epsilon}\p_{\mu}\tilde{\eta} dx .
\end{split}
\end{equation}
\emph{Step 2: Error bounds.}
Via Hölder's inequality and interpolation, we now address the error contributions which involve derivatives of $a^{ij}$ and which are not contained in $g$.
Using $n\geq 2$ and the smallness assumption (\ref{eq:small}), we, for instance, have:
\begin{align*}
\int\limits_{B_{1}^+}|\p_{\mu}a^{ij}||\p_j\ve||\p_{i\mu}\ve|\zeta^2 dx
&\leq  \Vert \partial_\mu a^{ij} \Vert_{L^{n+1}(A_{m}^+)} 
  \Vert \zeta \partial^2_{i\mu} \ve \Vert_{L^2(A_m^+)} \Vert \zeta \partial_j \ve \Vert_{L^{\frac{2(n+1)}{n-1}}(A_m^+)} 
\\ & \leq  C \delta  
  \Vert \zeta \partial^2_{i\mu} \ve \Vert_{L^2(A_m^+)} \Vert  \nabla( \zeta  \partial_j \ve )\Vert_{L^{2}(A_m^+)}.
 \end{align*}

Absorbing the second order contribution in the right hand side, allows us to deal with this and similar terms. 
We bound the terms involved in $ \int\limits_{B_1^+}|g^{\epsilon}||\p_{\mu}\tilde{\eta}|dx$:
By Hölder's inequality
\begin{align*}
\int\limits_{B_1^+}|g^{\epsilon}||\p_{\mu}\tilde{\eta}|dx 
&\leq  C(\delta)\int\limits_{B_1^+}|g^{\epsilon}|^2 \zeta^2 dx + \delta \int\limits_{B_1^+} \zeta^{-2}|\p_{\mu}\tilde{\eta}|^2 dx \\
& \leq C(\delta)\int\limits_{B_1^+}|g^{\epsilon}|^2 \zeta^2  dx + \delta \int\limits_{B_1^+}\zeta^2 |\p_{\mu}^2 \ve|^2 dx  + \delta \int\limits_{B_1^+} |(\nabla \zeta)\p_{\mu}\ve|^2 dx .
\end{align*}
Absorbing the second order derivative contribution into the left hand side of (\ref{eq:2der}), we arrive at
\begin{align*}
\Vert \zeta |\nabla^2 \ve|  \Vert_{L^2(B_1^+)}  &\leq  C ( \left\| (\nabla \zeta) \p_j \ve \right\|_{L^2(B_1^+)}^2 +  \left\| \zeta \p_j \ve \right\|_{L^2(B_1^+)}^2 +  \left\| \ve \right\|_{L^2(B_1^+)}^2 \\
& \quad + \Vert \zeta g^{\epsilon} \Vert_{L^2(B_1^+)}^2).
\end{align*}
We now pass to the limit $\epsilon \rightarrow 0$. Using the weak lower semi-continuity of the $L^2$ norm and the $L^2$ boundedness of $g$, we obtain
\begin{equation}
\label{eq:v_limit}
\begin{split}
\int\limits_{B_1^+}|\nabla^2 v|^2 \zeta^2 dx &\leq  C(p,n)( \left\| (\nabla \zeta) \p_j v \right\|_{L^2}^2  \\ 
& \quad  +  \left\| \zeta \p_j v \right\|_{L^2}^2 +  \left\| v \right\|_{L^2}^2) + \int\limits_{B_1^+}g^2 \zeta^2dx.
\end{split}
\end{equation}
Hence, it remains to bound $\left\| g \zeta \right\|_{L^2}$. Here we make use of the bound for $\phi$ and its derivatives in the dyadic annuli $A_m$ to which we can always localize by choosing $\zeta$ appropriately.
In a similar spirit as before, interpolation arguments and Hölder's inequality permit us to control $\int_{B_1^+} g^2\zeta^2 $ by
\begin{align*}
\int\limits_{B_1^+} g^2\zeta^2 dx \lesssim \tau^2 2^{2 m}\int\limits_{B_{1}^+}\zeta^2|\nabla v|^2 dx + \tau^4 2^{4 m}\int\limits_{B_{1}^+}\zeta^2 v^2dx + \int\limits_{B_1^+}\zeta^2 f^2 dx 
\end{align*}
As above the second order error contributions which arise on the right hand side can be absorbed in the left hand side of (\ref{eq:v_limit})\\

\emph{Step 4: Conclusion.}
Using the weak lower semi-continuity of the $L^2$ norm, we obtain the desired $H^2$ estimates for $v$.
\end{proof}

\section[Compactness, Homogeneity, Regularity]{Consequences of the Carleman Estimate: Compactness, Homogeneity and an Almost Optimal Regularity Result}
\label{sec:reg}
In this section we present the consequences of the Carleman estimate, Proposition \ref{prop:varmet}, and of Corollary \ref{cor:consequence_Carl}: These include the independence (of the chosen subsequence) of the rate of vanishing at free boundary points (c.f. Proposition~\ref{prop:indep}), the compactness of $L^2$-rescaled (sub)sequences of solutions of the thin obstacle problem (c.f. Proposition \ref{prop:blowup}), the homogeneity of the blow-up solutions along particular subsequences (c.f. Proposition \ref{lem:homo}) and an almost optimal regularity result (c.f. Proposition \ref{prop:almost}). The independence of the vanishing order on the subsequence along which the blow-up solution is considered, combined with the homogeneity of particular blow-up solutions allows to classify the lowest blow-up profile and to obtain growth estimates (c.f. Proposition \ref{prop:homo3} and Corollary \ref{cor:kappa}). \\

Apart from Proposition \ref{lem:homo} and Lemma \ref{lem:almosthom}, all the results of this section are entirely based on Corollary \ref{cor:consequence_Carl}. Only for the proof of Proposition \ref{lem:homo} do we have to invoke the full strength of the modified Carleman estimate from Lemma \ref{lem:varmet1}.\\
 
In the whole section, we assume that the assumptions (A0)-(A4) are satisfied; in particular,  all solutions of the thin obstacle problem (\ref{eq:varcoef}) are assumed to be ($L^2$) normalized and the uniformly elliptic, symmetric tensor field $a^{ij}$ is in $W^{1,p}$ with some $p\in(n+1,\infty]$. This ensures that the assumptions of both the Carleman estimate (\ref{eq:vCarl}) and of Corollary \ref{cor:consequence_Carl} are fulfilled throughout the section.

\subsection[Vanishing order]{The order of vanishing and growth estimates}

As a direct consequence of the Corollary \ref{cor:consequence_Carl}, we obtain information on the vanishing order of solutions of (\ref{eq:varcoef}). Here we use the following definition:

\begin{defi}[Vanishing order]
\label{defi:vanishing}
For any $u\in L^2(B_\delta^+(x_0))$, $\delta>0$, we define the vanishing order of $u$ at $x_0$ as
\begin{equation*}
\kappa_{x_0}:=\limsup_{r\rightarrow 0_+} \frac{\ln \left(\fint_{A^+_{r/2,r}(x_0)}u^2\right)^{1/2}}{\ln r} \in [-\infty,\infty].
\end{equation*} 
\end{defi}

With the aid of the vanishing order it is possible to quantify the growth of solutions of the thin obstacle problem. We first prove that the $\limsup$ in the definition of the vanishing order is in fact a limit. 

\begin{prop}\label{prop:indep}
Let $w$ be a solution of \eqref{eq:varcoef} in $B_1^+$ and $x_0\in B_1'$. Then
\begin{equation*}
\liminf_{r\rightarrow 0_+}\frac{\ln \left(\fint_{A^+_{r/2,r}(x_0)}w^2\right)^{1/2}}{\ln r}=\limsup_{r\rightarrow 0_+}\frac{\ln \left(\fint_{A^+_{r/2,r}(x_0)}w^2\right)^{1/2}}{\ln r}.
\end{equation*}
\end{prop}

\begin{proof}
Without loss of generality, we only show the result for $x_0=0$ and derive it as an immediate consequence of \eqref{eq:consequence_Carl1} in Corollary~\ref{cor:consequence_Carl}. Indeed, if $\kappa_0<\infty$, for any $\epsilon>0$ fixed, we set $\tau=(\kappa_0+(n-1)/2-\epsilon)/(1+c_0\pi/2)$ in \eqref{eq:consequence_Carl1}. Note that for the chosen $\tau$, 
\begin{equation}\label{eq:asymp}
\frac{\tau \tilde{\phi}(t)}{t}\rightarrow -(\kappa_0+\frac{n-1}{2}-\epsilon)\text{ as }t\rightarrow -\infty.
\end{equation} 
Considering the $\limsup_{r_1\rightarrow 0_+}$ on both sides of \eqref{eq:consequence_Carl1} and recalling our definition of the vanishing order, the right hand side of (\ref{eq:consequence_Carl1}) is bounded. Thus,
\begin{align*}
e^{\tau \tilde{\phi}(\ln r_2)}r_2^{-1}(1+|\ln r_2|)^{-1}\|w\|_{L^2(A^+_{r_2,2r_2})}\leq C \quad \text{for any }  0<r_2\ll r_3.
\end{align*}
Hence, taking the logarithm and dividing by $\ln r_2$, we have
\begin{align*}
\frac{\ln \|w\|_{L^2(A_{r_2,2r_2}^+)}}{\ln r_2}\geq -\frac{\tau \tilde{\phi}(\ln r_2)}{\ln r_2}+1 +\frac{\ln (1+|\ln r_2|)}{\ln r_2} +\frac{\ln C}{\ln r_2}.
\end{align*}
By \eqref{eq:asymp},
\begin{align*}
\liminf_{r_2\rightarrow 0_+}\frac{\ln \|w\|_{L^2(A_{r_2,2r_2}^+)}}{\ln r_2}\geq \kappa_0+\frac{n+1}{2}-\epsilon.
\end{align*}
Since $\epsilon$ is arbitrary, the above inequality implies the desired estimate.

For $\kappa_0=\infty$, we set $\tau=M$, for an arbitrarily large $M$, in \eqref{eq:consequence_Carl1} and argue similarly. 
\end{proof}

\begin{rmk}[Growth bounds]
\label{rmk:growth}
\begin{itemize}
\item The preceding proposition immediately leads to the following growth estimate:
For each $x_0\in B'_1$ and each $\epsilon>0$ there exists a radius $r_{\epsilon}= r_{\epsilon}(w)>0$ such that
\begin{align*}
r^{\kappa_{x_0}+\frac{n+1}{2} + \epsilon} \leq  \left\| w\right\|_{L^2(A^+_{r/2,r}(x_0))} \leq r^{\kappa_{x_0} +\frac{n+1}{2}- \epsilon} \mbox{ for all } 0<r \leq r_\epsilon.
\end{align*}
By virtue of the $L^2-L^{\infty}$ estimates, this also transfers to $L^{\infty}$ norms over balls.
\item We emphasize that these growth bounds are asymptotic bounds as $r \rightarrow 0$. The radii $r_{\epsilon}>0$ may vary from free boundary point to free boundary point. Hence, these are \emph{non-uniform} estimates in the free boundary points. In Lemma \ref{lem:growthimp}, we will indicate how to obtain related \emph{uniform} upper bounds.
\item While it is possible to obtain uniform (in the radius $r$) \emph{upper} bounds on the growth of solutions (c.f. Lemma \ref{lem:growthimp}), the \emph{lower} bounds can only be improved slightly (c.f. Corollary \ref{cor:growth1}). 
\end{itemize}
\end{rmk}

\begin{rmk}[Vanishing order]
\label{rmk:vanishing_order}
For points $x_0$ with $\kappa_{x_0}>-\frac{n+1}{2}$ (which guarantees integrability at zero), it would have been possible to consider a related notion of the vanishing order in which one considers the $L^2$ norms on (half) balls instead of (half) annuli:
\begin{align*}
\tilde{\kappa}_{x_0}:=\limsup_{r\rightarrow 0_+} \frac{\ln \left(\fint_{B_{r}^+(x_0)}u^2\right)^{1/2}}{\ln r}.
\end{align*}
However, at first sight this definition is less suited for our Carleman based arguments. Yet we remark that the two notions of vanishing order coincide. Indeed, by Remark \ref{rmk:growth} given any $\epsilon>0$, for $0<r<r_\epsilon = r_{\epsilon}(w)$ we have
\begin{align*}
r^{2\kappa_{x_0}+\epsilon+n+1}\sum_{j=0}^{\infty} 2^{-j(2\kappa_{x_0}+\epsilon+n+1)}
&\leq \|u\|_{L^2(B_r^+(x_0))}^2 \left( =\sum_{j=0}^\infty \|u\|_{L^2(A_{2^{-(j+1)}r,2^{-j}r}^+)}^2\right) \\
&\leq r^{2\kappa_{x_0}-\epsilon+n+1}\sum_{j=0}^{\infty} 2^{-j(2\kappa_{x_0}-\epsilon+n+1)}.
\end{align*}
If $-\frac{n+1}{2}<\kappa_{x_0}<\infty$, then the geometric series converge, thus $\tilde{\kappa}_{x_0}=\kappa_{x_0}$; if $\kappa_{x_0}=\infty$, then $\tilde{\kappa}_{x_0}=\infty$. On the other side,
if $-\frac{n+1}{2}<\tilde{\kappa}_{x_0}<\infty$, then it is not hard to see that there is a sequence $r_j\rightarrow 0$ such that
\begin{align*}
 \frac{\ln \left(\fint_{A^+_{r_j/2,r_j}}w^2\right)^{1/2}}{\ln (r_j)}  \leq C <\infty.
\end{align*}
Then by Proposition~\ref{prop:indep}, $\kappa_{x_0}$ is finite. It then necessarily equals to $\tilde{\kappa}_{x_0}$. If $\tilde{\kappa}_{x_0}=\infty$, then $\kappa_{x_0}=\infty$. This follows as otherwise $\kappa_{x_0}<\infty$ and hence $\tilde{\kappa}_{x_0}<\infty$, which is a contradiction. \\
\end{rmk}

Next we show the upper semi-continuity of the map
$
B'_1 \ni x\mapsto \kappa_x,
$
as a consequence of equation (\ref{eq:consequence_Carl1}) in Corollary \ref{cor:consequence_Carl}.

\begin{prop}[Upper semi-continuity]\label{prop:semi_cont}
Assume that $w$ is a solution of (\ref{eq:varcoef}).
Let $\kappa_{x}$ be the vanishing order of $w$ at $x\in B'_1$. Then the mapping 
$B'_1 \ni x\mapsto \kappa_{x}$
is upper semi-continuous.
\end{prop}

\begin{proof}
We only show the upper semi-continuity at the origin. For any other point the argument is analogous. We show that for any $\epsilon>0$ there exists $\delta>0$ such that $\kappa_x \leq \kappa_0 + \epsilon$ for all $|x|\leq \delta$. The statement is trivial if $\kappa_0=\infty$, thus we assume $\kappa_0<\infty$ ($\kappa_0>-\infty$ follows immediately as $w\in L^{2}(B_1^+)$).\\
Using the notation of Corollary \ref{cor:consequence_Carl}, we choose $\tau= (\kappa_0 + \frac{n-1}{2} + \epsilon)/(1+c_0\pi/2)$. Moreover, we consider radii $r_1,r_2,r_3$ such that
\begin{itemize}
\item[(a)] $r_3 = R_0/2 $, where $R_0$ is the radius from Corollary \ref{cor:consequence_Carl},
\item[(b)] $r_2=r_2(\epsilon,\kappa_0,n,x_0)>0$ is a fixed radius, which is chosen so small that 
$$e^{\tau \tilde{\phi}(\ln(2r_2))}(2r_2)^{-1}(1+(\ln (2r_2))^2)^{-1}\geq r_2^{-(\kappa_0+\frac{n+1}{2}+\frac{3\epsilon}{4})},$$ and that for sufficiently large $C=C(n,p)$ (which is possible by Remark~\ref{rmk:growth})
\begin{align}
\label{eq:grow1}
\left\| w \right\|_{L^2(A_{r_2/2,r_2}^+)} \geq C\frac{r_2^{\kappa_0+\frac{n+1}{2}+\frac{\epsilon}{2}}}{r_3^{\kappa_0+\frac{n+1}{2}+\epsilon}},
\end{align}
\item[(c)] $0< r_1 \ll r_2$. 
\end{itemize}
Let $x\in B^+_{r_2/4}$. Applying Corollary \ref{cor:consequence_Carl} with $x_0 = x$, the implies
\begin{align*}
r_2^{-((n+1)/2 + \kappa_0 + \frac{3\epsilon}{4})}\left\| w \right\|_{L^2(A^+_{r_2/4,2r_2}(x))} &\lesssim r_1^{-((n+1)/2 + \kappa_0 + \epsilon)}\left\| w\right\|_{L^2(A^+_{r_1,2r_1}(x))} \\
&\quad + r_3^{-((n+1)/2 + \kappa_0 + \epsilon)}\left\| w\right\|_{L^2(A_{r_3,2r_3}^+(x))}.
\end{align*}
Noticing that $A_{r_2/2,r_2}^+(0)\subset A_{r_2/4,2r_2}^+(x)$ for any $x\in B_{r_2/4}^+$, we switch from an annulus centered at $x$ to an annulus centered at $0$:
\begin{equation}
\label{eq:compare}
\begin{split}
r_2^{-((n+1)/2 + \kappa_0 + \frac{3\epsilon}{4})}\left\| w \right\|_{L^2(A^+_{r_2/2,r_2}(0))} & \lesssim r_1^{-((n+1)/2 + \kappa_0 + \epsilon)}\left\| w\right\|_{L^2(A^+_{r_1,2r_1}(x))}\\
& \quad + r_3^{-((n+1)/2 + \kappa_0 + \epsilon)}\left\| w\right\|_{L^2(A_{r_3,2r_3}^+(x))}.
\end{split}
\end{equation}
Due to the choice of $r_2$ in (\ref{eq:grow1}) and the $L^2$ normalization of $w$, the second term on the right hand side of (\ref{eq:compare}) can be absorbed in the left hand side of (\ref{eq:compare}).
As a consequence, we obtain
\begin{align*}
r_2^{-((n+1)/2 + \kappa_0 + \frac{3\epsilon}{4})}\left\| w \right\|_{L^2(A^+_{r_2/2,r_2}(0))} \lesssim r_1^{-((n+1)/2 + \kappa_0 + \epsilon)}\left\| w\right\|_{L^2(A^+_{r_1,2r_1}(x))},
\end{align*}
which -- using (\ref{eq:grow1}) -- we read as a lower bound for the right hand side (with $r_2$ fixed):
\begin{align}
\label{eq:vanord}
C \lesssim r_1^{-((n+1)/2 + \kappa_0 + \epsilon)}\left\| w\right\|_{L^2(A^+_{r_1,2r_1}(x))}.
\end{align}
Defining $\delta = r_2/4$, taking the logarithm of both sides of (\ref{eq:vanord}), dividing by $\ln(r_1)$ and passing to the limit $r_{1} \rightarrow 0$, yields the statement of the Proposition.
\end{proof}

The upper semi-continuity of the vanishing order immediately implies a lower bound for solutions around free boundary points: 

\begin{cor}[Growth estimates--lower bound]
\label{cor:growth1}
Let $w$ be a solution of (\ref{eq:varcoef}) in $B_{1}^+$. Let $K\subset B_{1}'$ be a compact set. Suppose that for some $\bar\kappa>0$, finite, $\kappa_x< \bar{\kappa}$ for all $x\in K$. Then there exist a neighborhood, $U$, of $K$ and a radius $r_0=r_0(K,\bar{\kappa},w)>0$ such that
\begin{align*}
\sup_{B_r^+(x)}|w|\geq r^{\bar\kappa} \mbox{ for all }0<r\leq r_0 \text{ and all }x\in U.
\end{align*}
\end{cor}

\begin{proof}
By the upper semi-continuity of the vanishing order (Proposition~\ref{prop:semi_cont}), there exists a neighborhood $U\subset B'_1$ of $K$ with $\overline{U}$ compact, such that $\kappa_x<\bar\kappa$ for any $x\in \overline{U}$. Then by Remark~\ref{rmk:vanishing_order} and by arguing similarly as in Proposition~\ref{prop:semi_cont}, we have:  For any $\bar x\in \overline{U}$, there exists $\delta=\delta(\bar x, \bar\kappa, w)>0$ such that
\begin{align*}
\left(\fint_{B_r^+(x)} w^2\right)^{1/2}\geq r^{\bar \kappa} \text{ for any  }x\in B'_\delta(\bar x)\text{ and } 0<r<\delta.
\end{align*}
By compactness of $\overline{U}$, there exists $r_0 =r_0(K,\bar\kappa,w)>0$ such that 
\begin{align*}
\left(\fint_{B_r^+(x)} w^2\right)^{1/2}\geq r^{\bar \kappa} \text{ for any } x\in \overline{U}\text{ and } 0<r<r_0.
\end{align*}
Since $\sup_{B_r^+(x)}|w|\geq (\fint_{B_r^+(x)}w^2)^{1/2}$, the conclusion follows.
\end{proof}

\begin{rmk}
We emphasize that the lower bound in Corollary \ref{cor:growth1} is a \emph{non-uniform} estimate in the solution $w$. In the next lemma we show a uniform \emph{upper} bound.
\end{rmk}

\begin{lem}[Uniform upper growth bounds]
\label{lem:growthimp}
Let $w$ be a solution of (\ref{eq:varcoef}) in $B_{1}^+$. Given a finite constant $\bar{\kappa}>0$, there exists a constant $C=C(\bar\kappa,n,p,\|\nabla a^{ij}\|_{L^p(B_1^+)})$ such that for all $x_0 \in \Gamma_w \cap B_{1/2}$ 
\begin{align*}
\sup\limits_{B_{r}^+(x_0)}|w| \leq C r^{\min\{\kappa_{x_0},\bar{\kappa}\}} |\ln(r)|^{2} \mbox{ for all } 0<r\leq R_0,
\end{align*}
where $R_0=R_0(\bar\kappa,n,p,\|\nabla a^{ij}\|_{L^p(B_1^+)})$ is the radius in Corollary~\ref{cor:consequence_Carl} with $\tau_0=\bar\kappa+n$.
\end{lem}

\begin{proof} 
We assume that $0\in \Gamma_w$ and will show the statement at $0$. For the other free boundary points it follows analogously.\\
\emph{Case 1:} $\kappa_0 \leq \bar{\kappa}$. 
By Proposition~\ref{prop:indep} for each $\epsilon >0$
\begin{align*}
\lim\limits_{\tilde{r} \rightarrow 0 } \tilde{r}^{-\left(\kappa_0 + \frac{n+1}{2} - \epsilon\right) } \left\|w  \right\|_{L^2(A_{\tilde{r},2 \tilde{r}}^+)} &= 0.
\end{align*}
This implies that given a sequence $\{\epsilon_j\}$, $\epsilon_j\rightarrow 0$, there exist corresponding radii $r_j\rightarrow 0$ such that 
\begin{align}
\label{eq:vanb}
r_j^{- \left(\kappa_0 + \frac{n+1}{2} - \epsilon_j\right) } \left\| w \right\|_{L^2(A_{r_j/2,r_j}^+)} \leq 1 \mbox{ for all } j\in \N.
\end{align}
Let $0<r<R_0/4$. We apply Corollary \ref{cor:consequence_Carl} with the radii $r_2=r$, $r_1=r_j$, with $j\in \N$ such that $0<r_j \ll r $, and $r_3 = R_0/2$. As the weighting factor we choose
\begin{align}
\label{eq:weight1}
\tau_j = \left(\kappa_0 + \frac{n-1}{2} - \epsilon_j \right) \left( 1 - c_0 \arctan(\ln r_j) + \frac{1}{2}c_0 \frac{\ln(1+\ln(r_j)^2)}{\ln(r_j)}  \right)^{-1},
\end{align}
i.e. $\tau_j$ is such that $e^{\tau_j \tilde{\phi}(\ln r_j)}r_j^{-1}=r_j^{-(\kappa_0+\frac{n+1}{2}-\epsilon_j)}$,
where $\epsilon_j$ is associated with $r_j$ as discussed in (\ref{eq:vanb}). Here $c_0>0$ is the constant from the Carleman inequality from Proposition \ref{prop:varmet} which will be determined later.
With these assumptions, Corollary \ref{cor:consequence_Carl} then yields
\begin{align*}
e^{\tau_j \tilde{\phi}(\ln(r))} r^{-1}(1+ \ln^2(r))^{-\frac{1}{2}} \left\| w \right\|_{L^2(A_{r,2r}^+)}\\
 \leq c_0^{-1}C(n,p)\left(1+e^{\tau_j\tilde{\phi}(\ln (R_0/2))}R_0^{-1}\|w\|_{L^2(A_{R_0/2,R_0})}\right).
\end{align*}
We note that $\tau_j\rightarrow (\kappa_0+\frac{n-1}{2})/(1+c_0\pi/2)$ as $j\rightarrow \infty$.
Passing to the limit $j\rightarrow \infty$ and using the explicit expression of $\tilde{\phi}$, implies the upper bound
\begin{equation}
\label{eq:upper_bound}
\begin{split}
r^{- \left(\kappa_0 + \frac{n-1}{2} \right) \frac{1-c_0\arctan(\ln(r))}{1 + c_0 \pi/2} -1} (1+\ln^2(r))^{-\frac{1}{2}-\frac{1}{2}\frac{c_0(\kappa_0+(n-1)/2)}{1+c_0\pi/2}} \left\|w \right\|_{L^2(A_{r, 2r}^+)} \\
\leq c_0^{-1}C(n,p)\left(1+R_0^{-(\kappa_0+\frac{n+1}{2})}\|w\|_{L^2(A_{R_0/2,R_0}^+)}\right).
\end{split}
\end{equation}
The asymptotics of the function $t \mapsto \arctan(t)$ yield 
$$\left| \frac{1-c_0\arctan(\ln(r))}{1 + c_0 \frac{\pi}{2}} \right| = 1 + \frac{1}{1+c_0\pi/2}\frac{1}{\ln(r)}+ o\left(\frac{1}{\ln(r)}\right).$$ 
Considering $c_0 = (\bar \kappa+(n-1)/2)^{-1}$, we then improve (\ref{eq:upper_bound}) to 
\begin{align*}
r^{- \left(\kappa_0 + \frac{n+1}{2} \right)} (1+\ln(r)^2)^{-1} \left\|w \right\|_{L^2(A_{r, 2r}^+)} \\
\leq C(n,p,\bar\kappa)\left(1+R_0^{-(\kappa_0+\frac{n+1}{2})}\|w\|_{L^2(A_{R_0/2,R_0}^+)}\right).
\end{align*}
Recalling the parameter dependence of $R_0$ and our assumption (A0), we obtain a constant $C=C(\bar\kappa,n,p,\|\nabla a^{ij}\|_{L^p(B_1^+)})$ such that 
\begin{align*}
r^{- \left(\kappa_0 + \frac{n+1}{2} \right)} (1+\ln(r)^2)^{-1} \left\|w \right\|_{L^2(A_{r, 2r}^+)}\leq C \quad \text{for all } 0<r<R_0,
\end{align*}
or in other words
\begin{align}
\label{eq:unibound_annulus}
r^{- \frac{n+1}{2}}\left\|w \right\|_{L^{2}(A_{r,2r}^+)} \leq C r^{\kappa_0} |\ln(r)|^{2}\quad \text{for all } 0<r<R_0.
\end{align}
This then amounts to
\begin{align*}
\sup_{B_r^+}|w| \leq C r^{\kappa_0} |\ln(r)|^{2},
\end{align*}
by arguing similarly as in Remark~\ref{rmk:vanishing_order} and using the $L^2-L^\infty$ estimate of the solution, i.e. $\sup_{B_{r/2}^+}|w|\leq Cr^{-(n+1)/2}\| w\|_{L^2(B_r^+)}$ (c.f. \cite{U87}).\\
\emph{Case 2:}  $\kappa_0 >\bar{\kappa}$. We define $\tau_j:= \left(\bar{\kappa} + \frac{n-1}{2} - \epsilon_j \right) \left( 1 - c_0 \arctan(\ln(r_j)) \right)^{-1}$ in (\ref{eq:weight1}) and argue similarly as in case 1.
\end{proof}

\subsection{Doubling and the blow-up procedure}

The next proposition is a central result of this section. It states that at a free boundary point with finite vanishing order, a solution, $w$, of (\ref{eq:varcoef}) satisfies an $L^2$-doubling estimate. This implies sufficient compactness properties in order to carry out a blow-up procedure.

\begin{prop}
\label{prop:doubling}
Let $w: B_1^+ \rightarrow \R$ be a solution of (\ref{eq:varcoef}). 
Then for all $x_0\in \Gamma_w \cap B_{1/2}'$ with $\kappa_{x_0}<\infty$, there exist a radius $r_{x_0}>0$ and a constant $C_{x_0}>0$ depending on $x_0,\gamma, \kappa_{x_0}, n, \left\|\nabla a^{i j} \right\|_{L^p(B_1^+)},w $, such that for all $0<r\leq r_{x_0}$
\begin{align*}
\int\limits_{B_{2r}^+(x_0)}w^2 dx  \leq C_{x_0} \int\limits_{B_{r}^+(x_0)}w^2 dx.
\end{align*}
\end{prop}

\begin{rmk}
We stress that, as we are using the growth estimate from Corollary \ref{cor:growth1}, the doubling constant $C_{x_0}>0$ and the radius $r_{x_0}>0$ are \emph{not} uniform in the respective free boundary point. Moreover, they strongly depend on the respective solution $w$.
\end{rmk}

We prove the doubling inequality as a consequence of Corollary \ref{cor:consequence_Carl}.

\begin{proof} 
Without loss of generality we prove the statement of Proposition \ref{prop:doubling} for $x_0 = 0$.
We start from Corollary \ref{cor:consequence_Carl} (using the second left hand side term with $r_1\sim r_2$) and choose $\epsilon>0 $ such that the growth estimate from Corollary \ref{cor:growth1} implies
\begin{align}
\label{eq:lowb}
\left\| w \right\|_{L^2(B_{\epsilon}^+)} \geq \epsilon^{\frac{n+1}{2}+\kappa_0 + \frac{1}{2}}.
\end{align}
Moreover, we set $\tau= \frac{n-1}{2} + \kappa_0 + 1$.
Then, using the notation from Corollary \ref{cor:consequence_Carl}, we obtain
\begin{align*}
e^{(\frac{n-1}{2}+\kappa_0 +1)\tilde{\phi}(\ln(2\epsilon))}\epsilon^{-1} \left\| w \right\|_{L^2(A_{\epsilon,2\epsilon}^+)} &\leq C( e^{(\frac{n-1}{2}+\kappa_0 +1)\tilde{\phi}(\ln(\epsilon/2))} \epsilon^{-1} \left\| w\right\|_{L^2(A_{\epsilon/2,\epsilon}^+)} \\
& \quad + e^{(\frac{n-1}{2}+\kappa_0 +1)\tilde{\phi}(\ln(r_3))}r_3^{-1} \left\| w\right\|_{L^2(A_{r_3, 2r_3}^+)}).
\end{align*}
Adding $e^{(\frac{n-1}{2}+\kappa_0 +1)\tilde{\phi}(2\epsilon)}\epsilon^{-1} \left\| w\right\|_{L^2(B_{\epsilon}^+)}$ to both sides of the inequality and invoking the monotonicity of $\tilde{\phi}$ yields
\begin{equation}
\label{eq:interm}
\begin{split}
e^{(\frac{n-1}{2}+\kappa_0 +1)\tilde{\phi}(\ln(2\epsilon))} \epsilon^{-1} \left\| w \right\|_{L^2(B_{2 \epsilon}^+)} &\leq C( e^{(\frac{n-1}{2}+\kappa_0 +1)\tilde{\phi}(\ln(\epsilon/2))}\epsilon^{-1} \left\| w\right\|_{L^2(B_{\epsilon}^+)} \\
& \quad + e^{(\frac{n-1}{2}+\kappa_0 +1)\tilde{\phi}(\ln(r_3))}r_3^{-1} \left\| w\right\|_{L^2(A_{r_3, 2r_3}^+)}).
\end{split}
\end{equation}
Using that, by the choice of $\epsilon$, the growth estimate (\ref{eq:lowb}) holds, and potentially making $\epsilon$ smaller by requiring $0<\epsilon \leq C \min\{\left\| w\right\|_{L^2(B_1^+)}^{-4},1\} r_3^{4\left(\frac{n+1}{2}+\kappa_0 +1\right)}$, we note that by (\ref{eq:lowb})
\begin{align*}
e^{(\frac{n-1}{2}+\kappa_0 +1)\tilde{\phi}(\ln(r_3))} r_3^{-1} \left\| w\right\|_{L^2(A_{r_3, 2r_3}^+)} \leq \frac{1}{2 C} e^{(\frac{n-1}{2}+\kappa_0 +1)\tilde{\phi}(\ln(2\epsilon))} \epsilon^{-1} \left\| w \right\|_{L^2(B_{2 \epsilon}^+)}.
\end{align*}
Thus, absorbing the second right hand side contribution of (\ref{eq:interm}) into the inequality's left hand side and using the explicit form of $\tilde{\phi}$ once more, results in
\begin{align*}
\left\| w \right\|_{L^2(B_{4\epsilon}^+)}
\leq C\left\|w\right\|_{L^2(B_{ 2\epsilon}^+)}.
\end{align*}
This is the desired doubling inequality for $r_0=\epsilon$. 
\end{proof}

The doubling property now provides sufficient compactness in order to carry out a blow-up procedure:

\begin{prop}[Blow-up limit]
\label{prop:blowup}
Let $w$ be a (non-trivial) solution of (\ref{eq:varcoef}) and let $0\in \Gamma_w$ with $\kappa_0<\infty$. 
Consider the rescaling 
\begin{align*}
w_{\sigma}(x) := \frac{w(\sigma x)}{\sigma^{-\frac{n+1}{2}} \left\| w\right\|_{L^2(B_{\sigma}^{+}(0))}}.
\end{align*}
Then along a sequence $\{ \sigma_j\}_{j\in \N }$,
\begin{align*}
w_{\sigma_j}(x) \rightarrow w_0(x) \mbox{ in } L^2(B_{1}^+) \mbox{ as } \sigma_j \rightarrow 0,
\end{align*}
and $w_0$ is a weak solution of the thin obstacle problem with constant coefficients:
\begin{equation*}
\begin{split}
\D w_0 &= 0 \mbox{ in } B_{1}^+,\\
w_0\geq 0, \ -\p_{n+1} w_0 \geq 0, \ w_0 (\p_{n+1} w_0) &= 0 \mbox{ on } B_{1}'.
\end{split}
\end{equation*}
Moreover, $\left\| w_0 \right\|_{L^2(B_1^+)}=1$; in particular, it is not the trivial function.
\end{prop}

\begin{proof}
The proof of the $L^2$ convergence relies on a compactness argument using the doubling property as well as the control on the contributions on the boundary $B_1'$. We present the details:
Using the gradient estimate, we obtain 
\begin{align*}
\left\| \nabla w \right\|_{L^2(B_{\sigma}^{+})}^2
\lesssim &\ \frac{1}{\sigma^2}\left\|  w \right\|_{L^2(B_{2\sigma}^{+})}^2 - \int\limits_{ B_{2\sigma}'}  w ( a^{n+1,j}\p_j w) d x,\\
\lesssim &\ \frac{1}{\sigma^2}\left\|  w \right\|_{L^2(B_{\sigma}^{+})}^2 ,
\end{align*}
where the last line is a consequence of the doubling inequality and the boundary conditions of the thin obstacle problem. In effect,
\begin{itemize}
\item $\left\| w_{\sigma}\right\|_{L^2(B_{1}^{+})} = 1$,
\item $\left\|  \nabla w_{\sigma}\right\|_{L^2(B_{1}^{+})} \leq C$.
\end{itemize} 
Hence, (along a not relabeled subsequence) we may pass to the limit $\sigma \rightarrow 0$ and obtain $w_{\sigma} \rightarrow w_0$ strongly in $L^{2}$ via Rellich's compactness theorem. 
Furthermore, as $a^{ij}(0)= \delta^{ij}$, $w_0$ weakly solves
\begin{align*}
\D w_0  & = 0 \mbox{ in } B_{1}^+,\\
w_0\geq 0, \ - \p_{n+1} w_0 \geq 0, \ w_0 (\p_{n+1} w_0) &= 0 \mbox{ on } B_{1}'.
\end{align*}
This proves the proposition.
\end{proof}

\begin{rmk}
\label{rmk:blowup}
Similarly as in Proposition \ref{prop:blowup} it is possible to define blow-up limits at any free boundary point $x_0\in \Gamma_w$:
\begin{align*}
w_{\sigma,x_0}(x) := \frac{w(\sigma (x-x_0))}{\sigma^{-\frac{n+1}{2}} \left\| w\right\|_{L^2(B_{\sigma}^{+}(x_0))}} \rightarrow w_{x_0}(x) \mbox{ in } L^2(B_{1}^+) \mbox{ as } \sigma \rightarrow 0.
\end{align*}
Similar as above, the corresponding blow-up limits are weak solutions of constant coefficient thin obstacle problems (however not necessarily with $a^{ij}= \delta^{ij}$).
\end{rmk}

\subsection{Homogeneous blow-ups}
\label{sec:homo}
In this section, we identify the minimal growth rates of solutions of (\ref{eq:varcoef}) close to free boundary points. For this we show that at each free boundary point there exists a blow-up sequence such that the corresponding blow-up limit is homogeneous (c.f. Proposition \ref{lem:homo}). Here Lemma \ref{lem:almosthom} plays an important role. Moreover, we show that the homogeneity is determined by the order of vanishing. Making use of the existence of homogeneous blow-up solutions as well as the classification of homogeneous global solutions (Proposition \ref{prop:homo2}), we then identify the lowest possible rate of vanishing (c.f. Proposition \ref{prop:homo3} and Corollary \ref{cor:kappa}).\\

The proof of the existence of homogeneous blow-ups is the only point in this section where we rely on the Carleman estimates from Section \ref{sec:Carl} directly (in the form of the modified Carleman estimate from Lemma \ref{lem:varmet1}), instead of only applying Corollary \ref{cor:consequence_Carl}. We have:

\begin{lem}[Almost homogeneity]
\label{lem:almosthom}
Let $w:B_1^+ \rightarrow \R$ be a solution of (\ref{eq:varcoef}) for which the assumptions (A0) and (A2) are satisfied. Let $\kappa_0>0$ be the vanishing order at $0$. 
Let $R_0=R_0(\kappa_0, n,p,\|\nabla a^{ij}\|_{L^p})$ be the radius determined by \eqref{eq:small1} in Lemma~\ref{lem:varmet1}.
Then we have the following statements:
\begin{itemize}
\item[(i)] Let $r_0, r_1$ be two radii with $0<r_0<r_1<R_0$ and let $R:=\ln(r_1/r_0)$. Assume that for some $\epsilon>0$,
\begin{align}\label{eq:poly_nondeg}
\left\| w \right\|_{L^2(B_{ 2r_0}^+)}\geq r_0^{\kappa_0 + \frac{n+1}{2}+ \frac{\epsilon}{64}}.
\end{align}
 Then if $R=R(r_1, \epsilon, n, \kappa_0)$ is large enough, there exists $r\in (r_0,r_1/2)$ such that 
\begin{align*}
\left\| x \cdot \nabla w - \kappa_0  w \right\|_{L^2(A_{r, 2r}^+)} \leq \epsilon \left\| w \right\|_{L^2(A_{r , 2r}^+)}.
\end{align*}
\item[(ii)] There exists a sequence of radii $r_j \rightarrow 0$ and $\epsilon_j\rightarrow 0$ such that
\begin{align*}
 \left\| x \cdot \nabla w - \kappa_0 w \right\|_{L^2(A_{r_{j}, 2 r_{j}}^+)} \leq \epsilon_j \left\| w \right\|_{L^2(A_{r_{j} , 2 r_{j}}^+)} \mbox{ for all } j\in \N.
\end{align*}
\end{itemize}
\end{lem}

\begin{proof}
\emph{Proof of (i):} We argue by contradiction.
Let $t_0:= \ln(r_0)$, $t_1:=\ln(r_1)$ and let $\bar{r}:= \sqrt{r_0 r_1}$ correspond to $\bar{t}:= \frac{t_0+t_1}{2}$. Note that $t_0+R=t_1$. \\
\emph{Claim:} If $R=R(t_1,n, \epsilon, \kappa_0)$ is large enough, then 
\begin{equation}\label{eq:bar_r}
\|w\|_{L^2(A_{\bar r,2\bar r}^+)}<  \bar r^{\kappa_0+\frac{n+1}{2}+\frac{\epsilon}{16}}.
\end{equation}
\emph{Proof of the claim:} Consider the following weight function:
\begin{equation}
\label{eq:weight_concave}
\tau \tilde{\phi}(t):= -\left( \kappa_0 + \frac{n-1}{2}-\frac{\epsilon }{16}\right) t + \frac{\epsilon}{2R}(t-t_0)(t_1-t). 
\end{equation}
We aim at applying this as a (non-convex) Carleman weight in the annulus $A_{r_0, r_1}^+$. For this we have to ensure that the conditions of Lemma \ref{lem:varmet1} are satisfied. As $\tilde{\phi}''''=0$, we only have to ensure conditions (a) and (b). Setting $\tau =  \kappa_0 + \frac{n-1}{2}-\frac{\epsilon}{16}$, condition (b) is satisfied, for $R=R(\kappa_0,n)$ chosen large enough. Condition (a) holds since
\begin{multline*}
\tau \tilde{\phi}'(t) = -\left( \kappa_0 + \frac{n-1}{2}-\frac{\epsilon }{16}\right) + \frac{\epsilon}{2R}(t_1-t)-\frac{\epsilon}{2R}(t-t_0)\\
\in \left[ -\kappa_0-\frac{n-1}{2}-\frac{3\epsilon }{16},   -\kappa_0-\frac{n-1}{2}+\frac{5\epsilon}{16} \right].
\end{multline*}
Thus, the assumptions (a)-(c) of Lemma \ref{lem:varmet1} are verified and Lemma \ref{lem:varmet1} is applicable (in particular in the form of equation \eqref{eq:consequence_Carl2}).
Consequently, we obtain 
\begin{align*}
\left\| e^{\tau \phi} |x|^{-1}w \right\|_{L^2(A_{\bar{r}, 2\bar{r}}^+)} \leq C\left( \left\| e^{\tau \phi}|x|^{-1} w \right\|_{L^2(A_{r_0, 2r_0}^+)} + \left\| e^{\tau \phi} |x|^{-1}w \right\|_{L^2(A_{r_1, 2r_1}^+)} \right),
\end{align*}
where $C=C(n,\epsilon, \kappa_0)$.
Inserting the specific form of the weight and the relation between $r_0,r_1$ (and $t_0, t_1$) then results in
\begin{equation}
\label{eq:nonconvexCarl}
\begin{split}
e^{\tau \tilde{\phi}( \bar t)}\bar r^{-1}\left\| w \right\|_{L^2(A_{\bar{r}, 2\bar{r}}^+)} &\leq C  \left[ e^{- \left( \kappa_0 + \frac{n+1}{2} -\frac{\epsilon}{16} \right)t_0} \left\| w \right\|_{L^2(A_{r_0, 2r_0}^+)} \right. \\
& \quad \quad \quad  \quad \quad \quad \left.
+ e^{- \left( \kappa_0 + \frac{n+1}{2} -\frac{\epsilon}{16} \right)t_1} \left\| w \right\|_{L^2(A_{r_1, 2r_1}^+)} \right].
\end{split}
\end{equation}
By Lemma~\ref{lem:growthimp}, the RHS of \eqref{eq:nonconvexCarl} is bounded. Hence, 
\begin{align*}
\|w\|_{L^2(A_{\bar r,2\bar r}^+)}\leq Ce^{-\tau \tilde{\phi}(\bar t)}\bar r& =Ce^{(\kappa_0+\frac{n+1}{2}-\frac{\epsilon}{16})\bar t -\frac{\epsilon R}{8}}\\
&=Ce^{(\kappa_0+\frac{n+1}{2}+\frac{\epsilon}{16})\bar t - \frac{\epsilon}{8}\bar t-\frac{\epsilon R}{8}}\\
&=Ce^{(\kappa_0+\frac{n+1}{2}+\frac{\epsilon}{16})\bar t -\frac{\epsilon}{8}(t_1-\frac{R}{2})-\frac{\epsilon R}{8}}\\
&=C\bar r^{\kappa_0+\frac{n+1}{2}+\frac{\epsilon}{16}} e^{-\frac{\epsilon}{8}t_1-\frac{\epsilon R}{16}}.
\end{align*}
Thus, if we take $R$ large enough such that
\begin{align*}
Ce^{-\frac{\epsilon}{8}t_1-\frac{\epsilon R}{16}}< 1, \quad \text{i.e. } R>\frac{16\ln C}{\epsilon}-2t_1,
\end{align*}
then we obtain \eqref{eq:bar_r}.

Now we apply Corollary~\ref{cor:consequence_Carl} with $\tau=(\kappa_0+(n-1)/2-\epsilon/64)/(1+c_0\pi/2)$ and the three radii $r_j\ll r_0<\bar r$:
\begin{align*}
&e^{\tau\tilde{\phi}(t_0)}r_0^{-1}|\ln r_0|^{-1}\|w\|_{L^2(A_{r_0,2r_0}^+)}\\
\quad &\leq C \left( r_j^{-(\kappa_0+(n+1)/2-\epsilon/64)}\|w\|_{L^2(A_{r_j,2r_j}^+)}+{\bar r}^{-(\kappa_0+(n+1)/2-\epsilon/64)}\|w\|_{L^2(A_{\bar r, 2\bar r}^+)}\right).
\end{align*}
In the limit $r_j\rightarrow 0$ the first term on the right hand side of the previous inequality vanishes due to Lemma~\ref{lem:growthimp}. Thus, together with \eqref{eq:bar_r} and enlarging the exponent to absorb the logrithmic term, we obtain
\begin{align*}
r_0^{-(\kappa_0+\frac{n+1}{2}-\frac{3}{2}\frac{\epsilon}{64})}\|w\|_{L^2(A_{r_0,2r_0}^+)}&\leq C\bar r^{\frac{5\epsilon}{64}}.
\end{align*}
Using a summation argument as in Remark \ref{rmk:vanishing_order} further yields 
\begin{align*}
r_0^{-(\kappa_0+\frac{n+1}{2}-\frac{3}{2}\frac{\epsilon}{64})}\|w\|_{L^2(B_{2r_0}^+)}&\leq C\bar r^{\frac{5\epsilon}{64}}.
\end{align*}
Recalling $\bar r=\sqrt{r_0r_1}$, implies
\begin{align*}
\|w\|_{L^2(B_{2r_0}^+)}&\leq C r_0^{\kappa_0+\frac{n+1}{2}-\frac{3}{2}\frac{\epsilon}{64}}\sqrt{r_0r_1}^{\frac{5\epsilon}{64}}\\
&< r_0^{\kappa_0+\frac{n+1}{2}+\frac{\epsilon}{64}},
\end{align*}
if $R$ is sufficiently large or $r_1$ is sufficiently small. This is a contradiction with \eqref{eq:poly_nondeg}.

\emph{Proof of (ii):}
Note that for given $\epsilon>0$, \eqref{eq:poly_nondeg} is satisfied for all $r\leq r_\epsilon(w)$. Thus (ii) follows from a contradiction argument and (i).
\end{proof}

Using the result of Lemma \ref{lem:almosthom} then yields the desired homogeneity result along certain subsequences:

\begin{prop}[Homogeneous blow-ups]\label{lem:homo}
Let $w:B_{1}^+ \rightarrow \R$ be a solution of the variable coefficient thin obstacle problem (\ref{eq:varcoef}). Let $x_0 \in \Gamma_w$. Then there exists a sequence of radii $\{r_j\}_{j\in \N}$ with $r_j \rightarrow 0$ as $j\rightarrow \infty$, which depends on $x_0$, i.e. $r_j = r_j(x_0)$, such that the associated ($L^2$-)blow-up sequence $\{w_{r_j,x_0}\}_{j\in \N}$ converges against a homogeneous solution, $w_{x_0}$, of a constant coefficient equation of the type (\ref{eq:constcoef}) (however not necessarily with $a^{ij}= \delta^{ij}$). The homogeneity of $w_{x_0}$ corresponds to the degree of vanishing, $\kappa_{x_0}$, of $w$ at $x_0$. 
\end{prop}

\begin{proof}
Without loss of generality we assume that $x_0$ is the origin. By Lemma \ref{lem:almosthom} (ii) there is a sequence $r_j \rightarrow 0$ such that
\begin{equation*}
\lim_{j\rightarrow \infty} \frac{\|x\cdot \nabla w-\kappa_0w\|_{L^2(A^+_{r_j, 2 r_{j}})}}{\|w\|_{L^2(A^+_{r_j, er_j})}}=0.
\end{equation*}
As $\left\| w \right\|_{L^2(A^+_{r_j, 2 r_{j}})} \leq \left\| w \right\|_{L^2(B_{2 r_j}^+)} $, this condition, in particular, enforces
\begin{equation*}
\lim_{j\rightarrow \infty} \frac{\|x\cdot \nabla w-\kappa_0w\|_{L^2(A^+_{r_j, 2 r_{j}})}}{\|w\|_{L^2(B_{2 r_j}^+)}}=0.
\end{equation*}
In terms of the blow-up sequence $w_{2  r_j}$ the above equation can then be rewritten as 
\begin{equation}\label{eq:homogeneous1}
\lim_{j\rightarrow \infty} \|x\cdot \nabla w_{2 r_{j}}-\kappa _0 w_{2  r_j}\|_{L^2(A_{1/2,1}^+)}=0.
\end{equation}
By Proposition~\ref{prop:blowup}, up to a subsequence, $w_{2 r_j}\rightarrow w_0$ in $C^{1,\alpha}_{loc}(\R^{n+1})$, where $w_0$ is a global solution to the thin obstacle problem \eqref{eq:constcoef}. Hence, \eqref{eq:homogeneous1} implies that 
\begin{equation*}
x\cdot \nabla w_0-\kappa_0 w_0=0\quad \text{ in } A_{1/2,1}^+.
\end{equation*}
Therefore $w_0$ is homogeneous in $A_{1/2,1}^+$. By analyticity (of solutions of the constant coefficient equation in the interior of $\R^{n+1}_+$), $w_0$ is homogeneous in $\R^{n+1}$. 
\end{proof}

Next we discuss the lowest possible homogeneities which appear in the previously described blow-up process. For this we recall (a slight modification of) Proposition 9.9 in \cite{PSU}, which characterizes global homogeneous solutions of (\ref{eq:constcoef}).

\begin{prop}[Proposition 9.9 in \cite{PSU}]
\label{prop:homo2}
Let $w_0$ be a homogeneous global solution of the thin obstacle problem \eqref{eq:constcoef} with $\kappa_0\in (1,2)$.
Then $\kappa_0=3/2$ and
\begin{equation*}
w_0(x)=\Ree(x_{n}+ix_{n+1})^{3/2}
\end{equation*}
up to multiplication by a constant and a rotation in $\R^n$.
\end{prop}

\begin{rmk}
We remark that the condition $\kappa_0 >1$ is always satisfied if $w_0$ is a blow-up limit, i.e. if there exists a blow-up sequence $w_{r_j}$, where $w$ solves \eqref{eq:varcoef} with $0\in \Gamma_w$, such that $w_{r_j}\rightarrow w_0$ in $C^{1,\alpha}_{loc}(\R^{n+1}_+)$. Indeed, since $0\in \Gamma_{w_{r_j}}$, we infer $w_{r_j}(0)=|\nabla w_{r_j}(0)|=0$ by the Signorini condition and the $C^{1,\alpha}_{loc}$ regularity of $w_{r_j}$. Here the identity $|\nabla' w_{r_j}(0)|=0$ follows by noting that the function $x' \mapsto w(x',0)$ is $C^{1,\alpha}(B_{1/2}')$ regular and attains a minimum at $x'=0$. The vanishing of the normal component of the gradient follows by approaching the free boundary point $0\in \Gamma_{w}$ from the interior of the open set $\Omega_{w_{r_j}}$ in combination with the $C^{1,\alpha}(B_1^+)$ regularity of $w_{r_j}$. Thus, passing to the limit and using the $C^{1,\alpha}_{loc}$ convergence of $w_{r_j}$ to $w_0$, yields $w_0(0)=|\nabla w_0(0)|=0$. This together with the $C^{1,\alpha}$ regularity of $w_0$ implies that $\kappa_0>1$ (more precisely $\kappa_0\geq 1+\alpha$). 
\end{rmk}

\begin{proof}
As the proof of Proposition \ref{prop:homo2} is well-known, we only give a sketch of it. Here we slightly deviate from the strategy presented in \cite{PSU}.\\
We recall that the key in the proof of Proposition 9.9 in \cite{PSU} is to show that for an arbitrary but fixed tangential direction, $e\in S^n\cap\{x_{n+1}=0\}$, the associated directional derivative, $\p_e w_0$, does not change its sign if $\kappa_0\in (1,2)$. This then entails that any solution $w_0$ only depends on two variables, where one is a tangential direction in $S^n\cap\{x_{n+1}=0\}$ and the other is the normal direction, $x_{n+1}$. In \cite{PSU} the sign condition for tangential derivatives is shown by making use of the Alt-Caffarelli-Friedman monotonicity formula \cite{ACF84}. \\
Indeed, this fact can also follow directly from the characterization of the second eigenspace for the Laplace-Beltrami operator on $S^n$: Extending $w_0$ evenly and letting $f$ denote the restriction of $\p_e w_0$ onto $S^n$, $f$ is an eigenfunction (with $\lambda=(\kappa_0-1)(\kappa_0+n-2)$) for $\Delta_{S^n}$ in $S^n\setminus \Lambda$ with $f=0$ on $\Lambda\cap S^n$. Let $0<\lambda_1(S^n\setminus \Lambda)<\lambda_2(S^n\setminus \Lambda)\leq \ldots, $ be the spectrum of this operator. If $\p_e w_0$ changes its sign, then by the variational formulation of the second eigenvalue of $\D_{S^n}$
\begin{equation}\label{eq:eigen}
\lambda\geq \lambda_2(S^n\setminus\Lambda)\geq \lambda_2(S^n).
\end{equation}
Since $\lambda_2(S^n)$ is realized by the affine functions $x_i$ with homogeneity $\kappa=1$, \eqref{eq:eigen} implies that $\kappa_0-1\geq \kappa=1$, i.e. $\kappa_0>2$, which yields a contradiction.
As a consequence, we infer that if $\kappa_0\in (1,2)$, then $w_0$ only depends on two variables, which, up to a rotation, we may assume to be $x_n$ and $x_{n+1}$. Recalling the complete characterization of homogeneous two-dimensional eigenfunctions (c.f. \cite{PSU}) of (\ref{eq:constcoef}), we hence infer that $w_0(x) = \Ree(x_{n}+ix_{n+1})^{3/2}$ up to a multiplicative constant. 
\end{proof}

Combining Proposition~\ref{lem:homo} and Proposition~\ref{prop:homo2} yields the following classification of the blow-up limits at free boundary points with vanishing order (strictly) less than two:

\begin{prop}\label{prop:homo3}
Let $w$ be a solution of \eqref{eq:varcoef} in $B_1^+$. Let $x_0\in B_{1}'$ be a free boundary point with $\kappa_{x_0}\in [0,2)$. Then $\kappa_{x_0}=3/2$ and there exists an ($L^2$) blow-up sequence $w_{r_j,x_0}$ such that $w_{r_j,x_0}\rightarrow w_{x_0}$ with $w_{x_0}(x)=w_{3/2}(B^{-1}(x_0)x)$. Here $w_{3/2}(x)= C_n Re(x\cdot \nu+ix_{n+1})^{3/2}$ for some unit vector $\nu\in S^{n-1}$; $B$ satisfies $A(x)=(a_{ij}(x))=B(x)B^t(x)$ and maps $\{x_{n+1}=0\}$ onto itself. 
\end{prop}

\begin{rmk}
\label{rem:homoblowup}
It is possible to obtain Proposition~\ref{prop:homo3} directly from Proposition~\ref{prop:indep} using the Friedland-Hayman inequality \cite{FH76} as follows (c.f. also \cite{An}):
Without loss of generality let $x_0=0$. Choose $\epsilon>0$ small enough such that $\kappa_0+2\epsilon\in(1,2)$.  By Corollary~\ref{cor:growth1}, there exists $r_0=r_0(\epsilon,w)$ such that
\begin{equation*}
r^{\kappa_0+\epsilon}<\left(\fint_{B_{r}^+}w^2\right)^{1/2}<r^{\kappa_0-\epsilon} \text{ for any }r\in(0,r_0).
\end{equation*}
Then it is not hard to check that for every $R>1$ and $r_jR<r_0$, one has
$$\left(\fint _{B_R^+} w_{r_j}^2\right)^{1/2}\leq R^{\kappa_0+2\epsilon}.$$
This implies that every blow-up limit $w_0$ has a less than quadratic growth rate at infinity, i.e. 
$$\sup _{B_{R/2}^+}|w_0|\leq \left(\fint _{B_R^+} w_{0}^2\right)^{1/2}\leq R^{\kappa_0+2\epsilon}  \text{ for any } R>1.$$
By the $C^{1,\alpha}$ estimates for solutions of (\ref{eq:varcoef}), 
$$\sup_{B_{R/4}^+}|\nabla w_0|\leq R^{\kappa_0 -1 +2\epsilon}  \text{ for any } R>1.$$
Up to an affine change of coordinates, $(\partial_e w)^\pm$ with $e\in \R^n$ are global subharmonic functions with less than linear growth at infinity. By the Friedland-Hayman inequality, this implies that $\partial_e w$ has a sign. The remaining arguments are the same as in Proposition 9.9 of \cite{PSU}.  
\end{rmk}

Proposition \ref{prop:homo3} immediately implies the following corollary:
\begin{cor}
\label{cor:kappa} Let $w:B_1^+ \rightarrow \R$ be a solution of (\ref{eq:varcoef}) and assume that $x_0\in \Gamma_w\cap B_{1}'$. Then $\kappa_{x_0}= \frac{3}{2}$ or $\kappa_{x_0}\geq 2$.
\end{cor}

\subsection{Almost optimal regularity}
In this section we follow a strategy described in the book by Petrosyan, Shahgholian and Uraltseva \cite{PSU} in order to obtain an almost optimal regularity result from our previous estimates: We begin by proving a in the free boundary points \emph{uniform} (almost optimal) $L^{\infty}$ growth estimate and then translate this into the desired (almost optimal) regularity result, c.f. Proposition \ref{prop:almost}.\\

As a corollary of Lemma \ref{lem:growthimp} we obtain the following growth estimate in an $R_0>0$ neighbourhood of the free boundary:

\begin{lem}[Growth at the free-boundary]
\label{lem:freeb_growth}
Let $w$ be a solution of the variable coefficient thin obstacle problem (\ref{eq:varcoef}) in $B_1^+(0)$. Consider $x\in B_{1/2}^+$ with $\dist(x,\Gamma_w)<\frac{1}{2}$. Then
\begin{align*}
 |w(x)| \leq C \dist(x,\Gamma_w)^{3/2 }\ln(\dist(x,\Gamma_w))^2,
\end{align*}
where $C=C(n, \left\|a^{ij}\right\|_{W^{1,p}}, p)$.
\end{lem}

\begin{proof}
This is an immediate consequence of the classification of the lowest possible vanishing order at free boundary points (Corollary \ref{cor:kappa}), the $L^{\infty}-L^2$ estimate and the uniformity of Lemma \ref{lem:growthimp} (applied with $\bar{\kappa}= 3/2$) in the free boundary points.
\end{proof}

With the growth estimate at hand we can infer an almost optimal regularity result which then also proves Proposition \ref{prop:almost_opti1} from the introduction:

\begin{prop}[Almost optimal regularity]
\label{prop:almost}
Let $w$ be a solution of (\ref{eq:varcoef}) in $B_{1}^+$.
\begin{itemize}
\item[(i)] If $a^{ij}\in W^{1,p}$ with $p\geq 2(n+1)$, then we have 
\begin{align*}
\left| \nabla w(x^1) - \nabla w(x^2) \right| \leq C \left| x^1-x^2\right|^{1/2}\ln(|x^1-x^2|)^2 
\end{align*}
$\mbox{ for all } x^1,x^2 \in B_{1/2}^+$.
\item[(ii)] If $a^{ij}\in W^{1,p}$ with $p\in (n+1,2(n+1))$, then
$w\in C^{1,\gamma}(B_{1/2}^+ \cup B_{1/2}')$ with
\begin{align*}
\left| \nabla w(x^1) - \nabla w(x^2) \right| \leq C(\gamma) \left| x^1-x^2\right|^{\gamma} \mbox{ for all } x^1,x^2 \in B_{1/2}^+,
\end{align*}
where $\gamma = 1-\frac{n+1}{p}$. 
\end{itemize}
Apart from the dependence on $\gamma$ in the second estimate, the constants $C>0$ in (i) and (ii) are functions of $n,\left\| w \right\|_{L^2(B_1^+)}, \left\|a^{ij}\right\|_{W^{1,p}}, p$.
\end{prop}

\begin{rmk}
\begin{itemize}
\item The statement (i) in particular implies that if $a^{ij}\in W^{1,p}$ with $p\geq 2(n+1)$, then for any $\epsilon\in (0,1/2)$ the function $w$ is $C^{1,1/2-\epsilon}(B_{1/2}^+ \cup B_{1/2}')$ regular (with a logarithmic loss as $\epsilon \rightarrow 0$).
\item The constant $C(\gamma)>0$ in (ii) diverges as $\gamma \searrow 0$ and as $\gamma \nearrow \frac{1}{2}$.
\end{itemize}
\end{rmk}

Up to the logarithmic loss, Proposition \ref{prop:almost} provides the optimal regularity for solutions of the thin obstacle problem. This follows, as on the one hand the regularity cannot exceed the regularity predicted by interior regularity estimates with $W^{1,p}$ coefficients. On the other hand $w_{3/2}(x):= \Ree(x_n + i x_{n+1})^{3/2}$ is a solution of the thin obstacle problem. 

\begin{proof}
We begin with the proof of (i) and follow the ideas in \cite{PSU}. For any $x\in B^+_{1/2}$ let $d(x)= \dist(x,\Gamma_w)$. Note that $B_{d(x)}(x)\cap \{x_{n+1}=0\}$ must be fully contained in either $\{w(\cdot, 0)=0\}$ or $\{w(\cdot, 0)>0\}$. 
Considering two points $x_1, x_2 \in B_{1/2}^+$ with $|x^1 -x^2|\leq \frac{R_0}{8}$, where $R_0>0$ is the radius from Lemma \ref{lem:freeb_growth}, we aim at showing
\begin{align}
\label{eq:est}
|\nabla w(x^1)- \nabla w(x^2)| \leq C|x^1-x^2|^{\frac{1}{2}}\ln(|x^1 - x^2|)^2.
\end{align}

For this purpose, we distinguish three cases: 
\begin{itemize}
\item[Case 1:] $d(x^1)\geq \frac{R_0}{4}$. Then $w$ satisfies an elliptic equation with Hölder coefficients in $B_{\frac{1}{4}}(x_1)$ with either Dirichlet or Neumann zero boundary conditions in the whole ball. Reflecting the metric, $a^{ij}$, and the solution, $w$, evenly or oddly with respect to $\{x_{n+1}=0\}$ and using that the assumption $a^{n+1j}(x',0)=0$ for all $j\in \{1,\dots,n\}$, permits us to work with interior elliptic estimates (as the resulting metric remains Hölder continuous). Hence, the desired estimate (\ref{eq:est}) is a direct consequence of this. In particular, the regularity of $w$ is determined by the regularity of the coefficients $a^{ij}$. In this case there is no logarithmic loss.

\item[Case 2:] $d(x^2)\leq d(x^1)\leq \frac{R_0}{4}$ and $|x^1-x^2|\geq \frac{d(x^1)}{2}$. Then $w$ solves an elliptic boundary value problem in the balls $B_{d(x^i)}(x^i)$, $i\in \{1,2\}$ (with either Dirichlet or Neumann data in the whole of  $B_{d(x^i)}(x^i)\cap B_{1}'$). Again we reflect both the solution and the metric evenly or oddly while preserving its Hölder regularity. By Lemma \ref{lem:freeb_growth} we obtain
\begin{align*}
|w(x)|\leq C d(x^i)^{\frac{3}{2}}\ln(d(x^i))^2 \mbox{ in } B_{d(x^i)}(x^i) \mbox{ for } i\in \{1,2\}.
\end{align*}
As a consequence, interior gradient estimates for elliptic equations with Hölder continuous coefficients imply
\begin{align*}
|\nabla w(x^i)| \leq C d(x^i)^{\frac{1}{2}}\ln(d(x^i))^2 \mbox{ for } i\in \{1,2\}.
\end{align*}
Thus,
\begin{align*}
|\nabla w(x^1)- \nabla w(x^2)| &\leq |\nabla w(x^1)| + |\nabla w(x^2)| \leq2C d(x^1)^{\frac{1}{2}}\ln(d(x^1))^2 \\
&\leq C|x^1- x^2|^{\frac{1}{2}}\ln(|x^1-x^2|)^2.
\end{align*}
\item[Case 3:] $d(x^2)\leq d(x^1)\leq \frac{R_0}{4}$ and $|x^1 - x^2|\leq \frac{d(x^1)}{2}$. As before we combine the growth estimate from Lemma \ref{lem:freeb_growth} with interior elliptic regularity estimates (after an appropriate reflection). In fact, from Lemma~\ref{lem:freeb_growth} we know
\begin{align*}
\sup_{B_{d(x^1)}(x^1)}|w(x)| &\leq Cd(x^1)^{\frac{3}{2}}\ln(d(x^1))^2.
\end{align*}
Thus by the interior elliptic regularity estimates 
\begin{align*}
\left\| w\right\|_{C^{1,\alpha}(B_{\frac{d(x^1)}{2}}(x^1))} 
&\leq C d(x^1)^{-1-\alpha}\left\|w \right\|_{L^\infty(B_{d(x^1)}(x^1))}\\
&\leq Cd(x^1)^{\frac{1}{2}-\alpha}\ln(d(x^1))^2,
\end{align*}
for $\alpha \in (0,1)$. We chose $\alpha=\frac{1}{2}$ (which is possible as $p\geq 2(n+1)$) and use $|x^1-x^2|<1$ (combined with the monotone decay of $\ln^2(t)$ for $0<t\leq 1$) to infer
\begin{align*}
|\nabla w(x^1) - \nabla w(x^2)| &\leq \left\| w\right\|_{C^{1,\alpha}(B_{\frac{d(x^1)}{2}}(x^1))}|x^1-x^2|^{\alpha}\\
& \leq C|x^1-x^2|^{\frac{1}{2}}\ln(|x^1-x^2|)^2.
\end{align*}
\end{itemize}
This proves the claimed inequality for (i). \\
For (ii) we argue analogously, noting that we can only apply interior regularity up to a $C^{1,\gamma}$ threshold. In the argument of case 2, we notice that in $B_{1/2}^+$ we have $|x^1-x^2|^{1/2}\ln(|x^1-x^2|)^2 \leq C(\gamma)|x^1-x^2|^{\gamma}$. In case 3 we consider Hölder estimates with $\alpha = \gamma$ and notice that $d(x^1)^{\frac{1}{2}-\gamma}\ln(d(x^1))^2 \leq C(\gamma)$. This completes the proof of Proposition \ref{prop:almost}.
\end{proof}

\bibliography{citations}
\bibliographystyle{alpha}
\end{document}